\documentclass[pdftex]{amsart}

\address{Research and Education Center for Natural Sciences, Keio University, 4-1-1 Hiyoshi, Kohoku-ku, Yokohama, Kanagawa, 223-8521, Japan}
\email{\href{mailto:isoshima@keio.jp}{isoshima@keio.jp}}

\address{Department of Mathematics, Faculty of Science, Okayama University, Okayama, 700-8530, Japan}
\email{\href{mailto:reo0713@s.okayama-u.ac.jp}{reo0713@s.okayama-u.ac.jp}}

\usepackage{amsmath,amssymb}
\usepackage{overpic}
\usepackage{amsfonts}
\usepackage{graphicx}

\usepackage{hyperref}
\usepackage{mathtools}
\DeclarePairedDelimiter{\abs}{\lvert}{\rvert}

\usepackage[hang,small,bf]{caption}
\usepackage[subrefformat=parens]{subcaption}
\usepackage{xcolor}
\usepackage{amsthm}
\usepackage{longtable}
\usepackage{bm}

\theoremstyle{plain}
\newtheorem{thm}{Theorem}[section]
\newtheorem{prop}[thm]{Proposition}
\newtheorem{lem}[thm]{Lemma}
\newtheorem{cor}[thm]{Corollary}
\newtheorem*{thm*}{Theorem}
\newtheorem*{cor*}{Corollary}
\newtheorem*{prop*}{Proposition}

\theoremstyle{definition}
\newtheorem{dfn}[thm]{Definition}
\newtheorem{rem}[thm]{Remark}
\newtheorem{exm}[thm]{Example}
\newtheorem{que}[thm]{Question}

\newtheorem{con}[thm]{Conjecture}
\newtheorem*{que*}{Question}
\newtheorem*{con*}{Conjecture}
\newtheorem*{nota*}{Notation}

\begin{document}

\title{Trisections and Lefschetz fibrations with $(-n)$-sections}

\author{Tsukasa Isoshima}
\author{Reo Yabuguchi}

\begin{abstract}
Castro and Ozbagci constructed a trisection of a closed 4-manifold admitting a Lefschetz fibration with a $(-1)$-section such that the corresponding trisection diagram can be explicitly constructed from a monodromy of the Lefschetz fibration. In this paper, for a closed 4-manifold $X$ admitting an achiral Lefschetz fibration with a $(-n)$-section, we construct a trisection of $X \# n\mathbb{C}P^2$ if $n$ is positive and $X \# (-n)\overline{\mathbb{C}P^2}$ if $n$ is negative such that the corresponding trisection diagram can be explicitly constructed from a monodromy of the Lefschetz fibration. We also construct a trisection of the fiber sum of two achiral Lefschetz fibrations with $n$- and $(-n)$-sections such that the corresponding trisection diagram can be explicitly constructed from monodromies of the Lefschetz fibrations.
\end{abstract}

\maketitle

\section{Introduction}

A trisection, which was introduced by Gay and Kirby \cite{MR3590351}, is a decomposition of a closed 4-manifold into three 4-dimensional 1-handlebodies. A trisection diagram that describes a closed 4-manifold is defined from a trisection, which consists of a closed surface and certain three cut systems on the surface. It is shown \cite{MR3590351} that every closed 4-manifold admits a trisection and a trisection diagram. However, it is generally difficult that one explicitly constructs a trisection diagram of a given 4-manifold. In this paper, we consider a method to construct explicit trisection diagrams via Lefschetz fibrations over $S^2$ with $(-n)$-sections, where $n$ is any integer. It is known \cite{MR3590351} that two closed 4-manifolds are diffeomorphic if and only if the corresponding trisection diagrams are related by surface diffeomorphisms, handle slides among the same family curves and (de)stabilizations.

A relative trisection, which was introduced by Castro \cite{castro2016relative}, is a decomposition of a 4-manifold with boundary. A relative trisection diagram that describes a 4-manifold with boundary is defined from a relative trisection. It is known \cite{castro2016relative} that an open book decomposition of the boundary of a 4-manifold with boundary is naturally induced from a relative trisection or a relative trisection diagram of the 4-manifold.

A genus-$g$ Lefschetz fibration over a surface $\Sigma$ is a smooth map 
$f \colon X \to \Sigma$ with finitely many critical points, each of which admits a 
complex local model $(z_1,z_2) \mapsto z_1 z_2$.  
Away from the critical values, $f$ is a surface bundle with fiber $\Sigma_g$, and each critical point 
determines a vanishing cycle on the fiber. The global topology of $X$ is encoded by a monodromy in the mapping class group $\mathrm{Mod}(\Sigma_g)$ into Dehn twists along these vanishing cycles.  
When both right- and left-handed Dehn twists appear in the monodromy, the Lefschetz fibration is said to be achiral.  
A section of self-intersection number $-n$ will be referred to as a $(-n)$-section.

Castro and Ozbagci \cite{MR3999550} constructed a trisection of a closed 4-manifold that admits a Lefschetz fibration over $S^2$ with a $(-1)$-section such that the corresponding trisection diagram can be drawn explicitly from a monodromy of the Lefschetz fibration. This construction is obtained as follows: First, decompose the 4-manifold into a regular neighborhood $V$ of the union of a regular fiber and a $(-1)$-section and the complement $W$ of $V$. Then, construct Lefschetz fibrations over $D^2$ for $V$ and $W$. After that, convert the Lefschetz fibrations into relative trisections. Finally, glue the relative trisections along the boundaries.

In this paper, we consider constructing a trisection of a closed 4-manifold $X$ that admits a Lefschetz fibration over $S^2$ with a $(-n)$-section for any integer $n$. Let $X=V \cup_{\partial}W$, where $V$ is a regular neighborhood of the union of a regular fiber and a $(-n)$-section and $W$ is the complement of $V$. If $n \not=1$, we cannot construct the trisection by the Castro and Ozbagci method since in this case, the induced open book decompositions of relative trisections of $V$ and $W$ obtained from their Lefschetz fibrations over $D^2$ are not compatible, and hence we cannot glue the relative trisections. Thus, we obtain the following theorem by decomposing the 4-manifold $X$ into just a neighborhood of a regular fiber and its complement. 
Here, let $n \mathbb{C}P^2$ denote $n \mathbb{C}P^2$ for a positive integer $n$, $(-n) \overline{\mathbb{C}P^2}$ for a negative integer $n$ and $S^4$ for $n=0$.

\begin{thm*}[Theorem \ref{thm:-n-section}]
Let $X$ be a closed 4-manifold admitting a genus-$p$ achiral Lefschetz fibration over $S^2$ with $m$ singular fibers and a $(-n)$-section, where $n$ is any integer. Then, $X \# n \mathbb{C}P^2$ admits a $(2p+m+\abs{n}+5, 2p+1)$-trisection whose corresponding trisection diagram can be constructed explicitly from a monodromy of the Lefschetz fibration of $X$.
\end{thm*}

Given genus-$g$ Lefschetz fibrations $f_i \colon X_i \to S^2$ with regular fibers 
$F \subset X_i$ for $i=1,2$, the fiber sum $X_1 \#_F X_2$ is obtained by removing 
tubular neighborhoods $\nu(F) \cong F \times D^2$ of $F$ from $X_1$ and $X_2$ and gluing 
the resulting boundaries $F \times S^1$ by a fiber-preserving diffeomorphism.  
The resulting 4-manifold again admits a genus-$g$ Lefschetz fibration over $S^2$, 
whose monodromy factorization is given by putting togather those of $X_1$ and $X_2$.
We partially construct a trisection of the fiber sum using a similar idea as Theorem \ref{thm:-n-section}.

\begin{thm*}[Theorem \ref{thm:fibersum}]
Let $X_i$ be closed 4-manifolds admitting genus-$p$ achiral Lefschetz fibrations over $S^2$ with $n_i$ singular fibers for $i=1,2$. Suppose that the Lefschetz fibrations of $X_1$ and $X_2$ have $(-n)$- and $n$-sections, respectively, where $n$ is any positive integer. Then, the fiber sum $X_1 \#_F X_2$ admits a $(2p+n_1+n_2+5, 2p+1)$-trisection whose corresponding trisection diagram can be constructed explicitly from monodromies of the Lefschetz fibrations of $X_1$ and $X_2$.
\end{thm*}

We will explicitly construct trisection diagrams of $E(n) \# n \mathbb{C}P^2$ and $E(1)_{2,3} \# 2 \mathbb{C}P^2$ after the proofs of the theorems (see Examples \ref{exm:E(n)} and \ref{exm:E(1)_2,3}). Note that the 4-manifold $E(1)_{2,3}$ is an exotic copy of $E(1)$. Here, for smooth 4-manifolds $X$ and $Y$, we say that $Y$ is an exotic copy of $X$ if $Y$ is homeomorphic to $X$ but not diffeomorphic to $X$.

\begin{que*}[Question \ref{que:CP^2}]
Can we explicitly construct trisection diagrams of 4-manifolds admitting achiral Lefschetz fibrations over $S^2$ with a $(-n)$-section from Theorem \ref{thm:-n-section} by removing $\abs{n}$ connected summands of $\mathbb{C}P^2$ or $\overline{\mathbb{C}P^2}$? In particular, what about exotic 4-manifolds such as $E(1)_{2,3}$ in Example \ref{exm:E(1)_2,3}?
\end{que*}

We say that a simply-connected closed 4-manifold $X$ is {\it almost completely decomposable} if $X \# \mathbb{C}P^2$ is diffeomorphic to $k \mathbb{C}P^2 \# \ell \overline{\mathbb{C}P^2}$ for some non-negative integers $k$ and $\ell$. If $X$ in Theorem \ref{thm:-n-section} is almost completely decomposable, we can obtain a trisection diagram of $k \mathbb{C}P^2 \# \ell \overline{\mathbb{C}P^2}$. 

\begin{cor*}[Corollary \ref{cor:ACD}]
Let $X$ be an almost completely decomposable 4-manifold admitting a genus-$p$ achiral Lefschetz fibration over $S^2$ with $m$ singular fibers and a $(-n)$-section, where $n$ is any positive integer. Then, for some non-negative integers $k$ and $\ell$, $k \mathbb{C}P^2 \# \ell \overline{\mathbb{C}P^2}$ admits a $(2p+m+n+5, 2p+1)$-trisection whose corresponding trisection diagram can be constructed explicitly from a monodromy of the Lefschetz fibration of $X$.
\end{cor*}

For example, $E(n) \# n \mathbb{C}P^2$ is diffeomorphic to $(3n-1) \mathbb{C}P^2 \# (10n-1) \overline{\mathbb{C}P^2}$ and $E(1)_{2,3} \# 2 \mathbb{C}P^2$ is diffeomorphic to $3 \mathbb{C}P^2 \# 9 \overline{\mathbb{C}P^2}$.

\begin{que*}[Question \ref{que:stabili}]
Is a trisection diagram constructed in Corollary \ref{cor:ACD} related to the standard trisection diagram of $k \mathbb{C}P^2 \# \ell \overline{\mathbb{C}P^2}$ shown in Figure \ref{fig:CP^2s} by surface diffeomorphisms, handle slides among the same family curves and destabilizations? In particular, what about trisection diagrams in Examples \ref{exm:E(n)} and \ref{exm:E(1)_2,3}?
\end{que*}

This paper is organized as follows: In Section \ref{sec:preliminaries}, we recall some notions used in the main theorems such as trisections and Lefschetz fibrations. In Section \ref{sec:main theorem}, we show the main theorems and explicitly construct several trisection diagrams for concrete examples of the main theorems. We also discuss the questions raised in the introduction.

\section*{Acknowledgement}

The authors would like to thank Nobutaka Asano and Hisaaki Endo for their helpful comments. The first author was partially supported by JSPS KAKENHI Grant Number JP25KJ0301. The second author was partially supported by JST OU-SPRING, and the Public Interest Incorporated Foundation ”Ohmoto ikueikai” in carrying out this research.

\section{Preliminaries}\label{sec:preliminaries}

In this paper, unless otherwise stated, each 4-manifold is compact, connected, oriented and smooth. The notation $X \cong Y$ indicates that two manifolds $X$ and $Y$ are diffeomorphic.

\subsection{Achiral Lefschetz fibrations}

In order to review (achiral) Lefschetz fibrations, we begin by recalling some basic notions about mapping class groups and open book decompositions.

Let $F$ be a compact oriented surface with (possibly empty) boundary. 
The \emph{mapping class group} $\operatorname{Mod}(F)$ is the group of isotopy classes of orientation–preserving diffeomorphisms of $F$ that restrict to the identity on $\partial F$. 

For a simple closed curve $\gamma \subset \operatorname{int}(F)$, the (right–handed) \emph{Dehn twist} about $\gamma$, denoted by $t_\gamma$, is represented by a diffeomorphism supported in an annular neighborhood of $\gamma$ that rotates once in the positive direction.
We compose mapping classes functionally: that is, $t_{\gamma_2}t_{\gamma_1}$ means apply $t_{\gamma_1}$ first, then $t_{\gamma_2}$.

An \emph{open book decomposition} of a closed oriented $3$-manifold $M$ is a pair $(B,\pi)$ consisting of a link $B \subset M$ (the \emph{binding}) and a fibration
\[
\pi \colon M \setminus B \to S^1
\]
such that each fiber $\pi^{-1}(\theta)$ is the interior of a compact surface $F_\theta$ with $\partial F_\theta = B$.
The surface $F=F_\theta$ is called a \emph{page} of the open book decomposition, and the isotopy class of the return map $F \to F$ (the \emph{monodromy}) determines the open book decomposition up to equivalence.

It is well known that every open book decomposition with page $F$ and monodromy $\phi$ determines an element $\phi \in \operatorname{Mod}(F)$, and conversely each element of $\operatorname{Mod}(F)$ determines an open book  decomposition by taking the mapping torus and gluing along $\partial F$. 

\begin{dfn}\label{def:achiral Lefschetz fibration}

Let $X$ be a 4–manifold (possibly with non–empty boundary) and $\Sigma$ a compact oriented surface (possibly with boundary).
A smooth map $f \colon X \to \Sigma$ is called a \textit{(possibly achiral) Lefschetz fibration} 
if there exist finitely many points $b_1,\dots,b_m \in \operatorname{int}(\Sigma)$ such that
$\mathrm{Critv}(f):=\{b_1,\dots,b_m\}$ is the set of critical values of $f$, and for each $i$ there exists a unique critical point $p_i \in f^{-1}(b_i)$ and oriented local complex coordinates $(z_1,z_2)$ in a neighborhood of $p_i$ and a complex coordinate $w$ in a neighborhood of $b_i$, compatible with the orientations of $X$ and $\Sigma$, respectively, such that
\[
f(z_1,z_2)=z_1^2+z_2^2 \quad \text{(chiral singularity)}
\]
or
\[
f(z_1,z_2)=z_1^2+\overline{z}_2^{\,2} \quad \text{(achiral singularity)}.
\]
\end{dfn}
It follows that
\[
f\big|_{\,f^{-1}(\Sigma\setminus\{b_1,\dots,b_m\})}\colon f^{-1}\!\big(\Sigma\setminus\{b_1,\dots,b_m\}\big)\to \Sigma\setminus\{b_1,\dots,b_m\}
\]
is a smooth fiber bundle with fiber a compact oriented surface $F$, possibly with boundary (see, \cite{Ehr51}). 
In particular, when $\Sigma=D^2$, the restriction $f\big|_{\,f^{-1}(\partial D^2)}$ is an $F$-bundle over $S^1$; consequently $\partial X$ carries an open book decomposition with binding $\partial F$ and page $F$.

Moreover, if a $4$-manifold $X$ admits a Lefschetz fibration $f\colon X\to D^2$ with regular fiber $F$, then the restriction $f|_{f^{-1}(\partial D^2)}$ induces an open book decomposition of $\partial X$ with page $F$ and monodromy given by the product of Dehn twists corresponding to vanishing cycles of $f$.

In this definition, the regular fiber $F$ may have boundary. 
When we refer to a \emph{genus–$p$ Lefschetz fibration}, 
we mean that the regular fiber is a closed surface of genus $p$; 
in the case where the fiber has non–empty boundary, we will always specify the regular fiber explicitly.

This definition follows Akbulut-Ozbagci~\cite{AO01}, 
extended to allow achiral singularities as in Gompf-Stipsicz~\cite{GS99}.
For background on Stein surface constructions see Gompf~\cite{Gompf98}, 
and for the relation with open book decompositions see Giroux~\cite{Giroux02}.

\subsection{Monodromy representations and factorizations}

Let $f\colon X\to\Sigma$ be a (possibly achiral) Lefschetz fibration with regular fiber $F$.
It is known \cite{MR1555438, MR881797} that the mapping class group $\operatorname{Mod}(F)$ is generated by Dehn twists about simple closed curves in $F$. 
Let 
$\mathrm{Critv}(f):=\{b_1,\dots,b_m\}\subset\operatorname{int}(\Sigma)$ be the set of critical values of $f$.
Fix a basepoint $*\in\Sigma\setminus \mathrm{Critv}(f)$ and an identification of the fiber $F_*:=f^{-1}(*)$ with $F$ by an orientation–preserving diffeomorphism that is the identity on $\partial F$.
Parallel transport along loops based at $*$ defines the \emph{monodromy representation}
\[
\rho\;\colon\;\pi_1(\Sigma\setminus \mathrm{Critv}(f),*)\to \operatorname{Mod}(F),
\]
where $\operatorname{Mod}(F)$ denotes the mapping class group of $F$ fixing $\partial F$ pointwise.

\begin{dfn}
A \emph{distinguished system of vanishing paths} is a collection of embedded arcs 
$\alpha_1,\dots,\alpha_m\subset\Sigma$ with endpoints in $\{*,b_1,\dots,b_m\}$ such that:
\begin{itemize}
  \item each $\alpha_i$ joins $*$ to $b_i$ and is disjoint from $\mathrm{Critv}(f)$ in its interior,
  \item the interiors of the $\alpha_i$ are pairwise disjoint, meeting only at $*$,
  \item the $\alpha_i$ are ordered according to the positive (counterclockwise) order of their tangent directions at $*$.
\end{itemize}
\end{dfn}

For each $i$, parallel transport of $F_*$ along $\alpha_i$ determines a properly embedded \emph{vanishing thimble} 
$\Delta_i\subset X$ and a simple closed curve $\gamma_i\subset F$ called the \emph{vanishing cycle}. 
If the critical point over $b_i$ is chiral (positive), the local monodromy is $t_{\gamma_i}$; 
if it is achiral (negative), the local monodromy is $t_{\gamma_i}^{-1}$.
We encode this by an exponent $\varepsilon_i\in\{+1,-1\}$ so that the local monodromy is $t_{\gamma_i}^{\varepsilon_i}$.

Choosing a small circle $\ell$ around $\mathrm{Critv}(f)$ that is positively oriented and based at $*$, and denoting by 
$\mu_i\in\pi_1(\Sigma\setminus \mathrm{Critv}(f),*)$ the standard simple loop that winds once positively around $b_i$ and no other $b_j$, in the order prescribed by the distinguished system, one has
\[
\rho(\mu_i)=t_{\gamma_i}^{\varepsilon_i}\in \operatorname{Mod}(F) \qquad (i=1,\dots,m).
\]
Hence the global monodromy along $\ell$ is the product
\[
\rho(\ell)\;=\;\rho(\mu_m)\cdots\rho(\mu_1)\;=\;t_{\gamma_m}^{\varepsilon_m}\cdots t_{\gamma_1}^{\varepsilon_1}\in \operatorname{Mod}(F).
\]
In particular, if $\Sigma=D^2$, then $\ell=\partial D^2$ and the induced open book  decomposition on $\partial X$ has monodromy
\[
\phi\;=\;t_{\gamma_m}^{\varepsilon_m}\cdots t_{\gamma_1}^{\varepsilon_1}\in \operatorname{Mod}(F).
\]
If $\Sigma=S^2$, then $\ell$ is null–homotopic in $\Sigma\setminus \mathrm{Critv}(f)$ and therefore 
\[
t_{\gamma_m}^{\varepsilon_m}\cdots t_{\gamma_1}^{\varepsilon_1}\;=\;\operatorname{id}\in \operatorname{Mod}(F).
\]

\begin{dfn}\label{def:Monodromy factorization}
A \emph{monodromy factorization} (with achiral singularities allowed) of a mapping class $\phi\in \operatorname{Mod}(F)$ is a word
\[
\phi\;=\;t_{\gamma_m}^{\varepsilon_m}\cdots t_{\gamma_1}^{\varepsilon_1} \ 
(\gamma_i\subset F\text{: simple closed curves},\ \varepsilon_i\in\{+1,-1\})
\]
arising from a Lefschetz fibration $f\colon X\to D^2$ with regular fiber $F$ as above. 
When all $\varepsilon_i=+1$, the factorization is called \emph{positive}.
\end{dfn}

Different choices of distinguished systems and identifications produce equivalent factorizations in the following sense.

\begin{dfn}\label{def:Hurwitz equivalence}
Two factorizations 
$t_{\gamma_m}^{\varepsilon_m}\cdots t_{\gamma_1}^{\varepsilon_1}$ and 
$t_{\gamma'_m}^{\varepsilon'_m}\cdots t_{\gamma'_1}^{\varepsilon'_1}$ in $\operatorname{Mod}(F)$ are \emph{Hurwitz equivalent} if one can be obtained from the other by a finite sequence of the following \emph{Hurwitz moves}: for some $i$,
\[
\big(\,\cdots,\;t_{\gamma_{i+1}}^{\varepsilon_{i+1}},\;t_{\gamma_{i}}^{\varepsilon_{i}},\;\cdots\,\big)
\;\longmapsto\;
\big(\,\cdots,\;t_{\gamma_{i}}^{\varepsilon_{i}},\;t_{t_{\gamma_{i}}^{\varepsilon_{i}}(\gamma_{i+1})}^{\varepsilon_{i+1}},\;\cdots\,\big),
\]
where $t_{\gamma}^{\pm1}$ denotes the (positive/negative) Dehn twist about $\gamma$.
\end{dfn}

\begin{dfn}\label{def:global conjugation}
Two factorizations $t_{\gamma_m}^{\varepsilon_m}\cdots t_{\gamma_1}^{\varepsilon_1}$ and $t_{\gamma_m'}^{\varepsilon_m'}\cdots t_{\gamma_1'}^{\varepsilon_1'}$ in $\operatorname{Mod}(F)$,
where each $\varepsilon_i,\varepsilon_i' \in \{+1,-1\}$, are \emph{globally conjugate} if there exists
$h\in \operatorname{Mod}(F)$ such that, for every $i=1,\dots,m$,
\[
t_{\gamma_i'}^{\varepsilon_i'} \;=\; h\, t_{\gamma_i}^{\varepsilon_i}\, h^{-1}.
\]
Equivalently,
\[
h\Big(t_{\gamma_m}^{\varepsilon_m}\cdots t_{\gamma_1}^{\varepsilon_1}\Big)h^{-1}
=
\big(h t_{\gamma_m}^{\varepsilon_m} h^{-1}\big)\cdots
\big(h t_{\gamma_1}^{\varepsilon_1} h^{-1}\big).
\]
\end{dfn}

\begin{prop}
Up to isomorphism a Lefschetz fibration $f\colon X\to D^2$ with regular fiber $F$ determines a monodromy factorization of $\phi\in \operatorname{Mod}(F)$ that is unique up to Hurwitz equivalence and global conjugation.
Conversely, any factorization in $\operatorname{Mod}(F)$ of the form 
$\phi=t_{\gamma_m}^{\varepsilon_m}\cdots t_{\gamma_1}^{\varepsilon_1}$ gives rise to a (possibly achiral) Lefschetz fibration over $D^2$ with regular fiber $F$ and vanishing cycles $\gamma_i$ of signs $\varepsilon_i$.
\end{prop}

\subsection{$(-n)$-sections and monodromy factorizations}

Suppose that a Lefschetz fibration $f\colon X\to \Sigma$ with regular fiber $F$ admits a section $S$ of self-intersection $-n$.
Then the boundary $\partial \nu(S)$ of a tubular neighborhood of $S$ is diffeomorphic to the circle bundle over $\Sigma$ with Euler number $-n$.
In particular, when $\Sigma=D^2$, the boundary component corresponding to $S$ is the circle bundle over $D^2$ restricted to its boundary $S^1$, with Euler number $-n$.

On the level of mapping class groups, this condition forces the global monodromy around $\partial\Sigma$ to lie in the subgroup of $\operatorname{Mod}(F)$ that fixes a chosen boundary component of $F$ and acts as the $n$-fold Dehn twist about that boundary component.
Equivalently, if we regard $F$ as a surface with one distinguished boundary component $\delta$ corresponding to the section, then any monodromy factorization of $f$ has the form
\[
t_{\gamma_m}^{\varepsilon_m}\cdots t_{\gamma_1}^{\varepsilon_1}\;=\;t_\delta^n \in \operatorname{Mod}(F),
\]
where $\gamma_i\subset F$ are vanishing cycles and $\varepsilon_i\in\{+1,-1\}$ record the chirality of the singularities.
Thus the presence of a $(-n)$-section is encoded in the monodromy factorization as the boundary multitwist $t_\delta^n$.

In particular:
\begin{itemize}
  \item If $n=1$, a $(-1)$-section corresponds to the relation
  \[
  t_{\gamma_m}^{\varepsilon_m}\cdots t_{\gamma_1}^{\varepsilon_1} = t_\delta
  \]
  in $\operatorname{Mod}(F)$.
  Such factorizations arise naturally when blowing up base points of Lefschetz pencils, since the exceptional spheres give disjoint $(-1)$-sections.
  \item More generally, a collection of disjoint $(-1)$-sections gives rise to a product of boundary twists, one for each section, in the global monodromy relation; for example, $t_{\gamma_m}^{\varepsilon_m}\cdots t_{\gamma_1}^{\varepsilon_1} = t_{\delta_k}\cdots t_{\delta_1}$.
\end{itemize}

\subsection{Fiber sums and twisted fiber sums}

Let $f_i \colon X_i \to \Sigma_i$ ($i=1,2$) be (possibly achiral) Lefschetz fibrations 
with the same regular fiber $F$, where $F$ is a compact oriented surface (possibly with boundary).
Fix identifications of the fibers with a reference surface $F_0$ by orientation–preserving diffeomorphisms
\[
\phi_i \colon F \;\;\xrightarrow{\;\cong\;}\; F_0.
\]

\begin{dfn}\label{def:fiber-sum}
The \emph{fiber sum} $X_1 \#_F X_2$ is obtained by choosing orientation–reversing diffeomorphisms
\[
\psi \colon \nu(F)\subset X_1 \;\;\xrightarrow{\;\cong\;}\;\nu(F)\subset X_2
\]
between tubular neighborhoods of regular fibers, which identify the fibers via the fixed diffeomorphisms $\phi_i$.
Removing the interiors of $\nu(F)$ from each $X_i$ and gluing along their common boundary via $\psi$
produces a new Lefschetz fibration
\[
f_1 \#_F f_2 \colon X_1 \#_F X_2 \;\to\; \Sigma_1 \# \Sigma_2
\]
with regular fiber $F_0$.
\end{dfn}

In this definition, the gluing map $\psi$ is assumed to restrict to the identity on the fiber $F_0$.  
More generally, one may allow a nontrivial element of the mapping class group of $F_0$.

\begin{dfn}\label{def:twisted-fiber-sum}
Given Lefschetz fibrations $f_i \colon X_i \to \Sigma_i$ ($i=1,2$) with regular fiber $F$,
and an element $h\in \operatorname{Mod}(F)$, the \emph{twisted fiber sum}
\[
X_1 \#_h X_2
\]
is obtained as above, except that the gluing map $\psi$ identifies the fibers by $h\colon F_0\to F_0$.
The resulting fibration $f_1 \#_h f_2$ has base $\Sigma_1 \# \Sigma_2$ and regular fiber $F_0$.
\end{dfn}

Note that in the case $h=\operatorname{id}$ we recover the ordinary fiber sum.
In particular, when the bases are $S^2$, the base of the resulting fibration is again $S^2$,
and the global monodromy factorization is given by the concatenation of the factorizations of $f_1$ and $f_2$,
up to global conjugation by $h$.
Here we record the standard convention for the action of $h$ on a
monodromy factorization.
If
\[
W = t_{\gamma_k}\cdots t_{\gamma_1}
\]
is a word of Dehn twists, we define
\[
h(W) := t_{h(\gamma_k)}\cdots t_{h(\gamma_1)}.
\]
Thus $h$ acts on $W$ by applying $h$ to each vanishing cycle.
Note that this is not conjugation, but the push-forward
of vanishing cycles under $h$.
This is the standard convention in the context of twisted fiber sums.
In particular, if \(W_i\) \((i=1,2)\) are monodromy factorizations of
\(f_i\), then the monodromy of the twisted fiber sum 
\(f_1 \#_{h} f_2\) is
\[
h(W_1)\cdot W_2.
\]

\subsection{Handle diagrams arising from monodromy factorizations of Lefschetz fibrations over $D^2$}\label{subsec:handle diagram}
We recall the standard procedure to recover a handle diagram of the total space $X$
from a (possibly achiral) monodromy factorization
\[
\phi \;=\; t_{\gamma_m}^{\varepsilon_m}\cdots t_{\gamma_1}^{\varepsilon_1}\in \operatorname{Mod}(F)\ 
(\gamma_i\subset F \text{ simple closed curve},\ \varepsilon_i\in\{\pm1\})
\]
of a Lefschetz fibration $f\colon X\to D^2$ with regular fiber $F$.

\subsubsection*{Step 0 (Construct the base handlebody).}
Start from the product $F\times D^2$.
In handle calculus this is drawn as one $0$-handle together with $1$-handles encoding $F$:
if $F=\Sigma_{p,b}$ (genus $p$ with $b$ boundary components), take $2p+b-1$ dotted $1$-handles (``Akbulut-Kirby'' page presentation).
The boundary $\partial(F\times D^2)$ carries the open book decomposition with page $F$ and trivial monodromy \cite[§§5.4, 8]{GS99}, \cite[§1]{AO01}.

\subsubsection*{Step 1 (Place the attaching circles from vanishing cycles).}
For each vanishing cycle $\gamma_i\subset F$ appearing in the factorization, realize the attaching circle
\[
\Lambda_i\ :=\ \gamma_i\times\{p_i\}\ \subset\ \partial(F\times D^2),
\]
on a chosen page copy of $F\subset \partial(F\times D^2)$, where $p_i\in\partial D^2$ are pairwise distinct points (only used to separate the attachments).
Draw $\Lambda_i$ on the page so that it represents the isotopy class of $\gamma_i$ in $F$.

\subsubsection*{Step 2 (Assign framings).}
Attach a $2$-handle along each $\Lambda_i$ with framing \emph{relative to the page framing}
\[
\operatorname{fr}(\Lambda_i)=
\begin{cases}
-1, & \text{if } \varepsilon_i=+1 \ \text{(positive/chiral singularity)},\\[2pt]
+1, & \text{if } \varepsilon_i=-1 \ \text{(negative/achiral singularity)}.
\end{cases}
\]
Equivalently, in the usual handle diagram the coefficient is $-1$ for positive nodes and $+1$ for negative nodes \cite[§§8.1-8.3]{GS99}, \cite[§1]{AO01}.

\subsubsection*{Step 3 (Boundary open book decomposition and monodromy).}
The handlebody obtained from $F\times D^2$ by these $2$-handle attachments is diffeomorphic to the total space $X$.
Its boundary inherits the open book decomposition whose monodromy is the product
$t_{\gamma_m}^{\varepsilon_m}\cdots t_{\gamma_1}^{\varepsilon_1}\in \operatorname{Mod}(F)$.

\subsubsection*{Step 4 (Interpret Hurwitz moves as Kirby moves).}
Hurwitz moves on the factorization correspond to Kirby moves:
\begin{itemize}
  \item The elementary Hurwitz move
  $(\cdots,t_{\gamma_{i+1}}^{\varepsilon_{i+1}},t_{\gamma_{i}}^{\varepsilon_{i}},\cdots)
  \mapsto
  (\cdots,t_{\gamma_{i}}^{\varepsilon_{i}},t_{t_{\gamma_{i}}^{\varepsilon_{i}}(\gamma_{i+1})}^{\varepsilon_{i+1}},\cdots)$
  is implemented by sliding the $2$-handle on $\Lambda_i$ over that on $\Lambda_{i+1}$.
  \item Global conjugation by $h\in \operatorname{Mod}(F)$ is realized by a diffeomorphism of the page $F$ extending over $F\times D^2$ (ambient isotopy of the diagram).
\end{itemize}
Thus the handle diagram is well defined up to handle slides and isotopies exactly in the indeterminacy of Hurwitz equivalence and global conjugation.

\subsection{Trisections and relative trisections}

We first recall the definition and some properties of trisections of closed 4-manifolds that were introduced by Gay and Kirby \cite{MR3590351} as a 4-dimensional analogue of Heegaard splittings.

\begin{dfn}
Let $X$ be a closed 4-manifold and $g$ and $k_i$ non-negative integers with $k_i \le g$ for $i=1,2,3$. A $(g;k_1,k_2,k_3)$-\textit{trisection} of $X$ is a 3-tuple $(X_1,X_2,X_3)$ satisfying the following conditions:
\begin{itemize}
\item $X=X_1 \cup X_2 \cup X_3$,
\item For each $i=1,2,3$, $X_i \cong \natural_{k_i} S^1 \times D^3$,
\item For each $i=1,2,3$, $X_i \cap X_{i+1} \cong \natural_{g} S^1 \times D^2$, where $X_4=X_1$ and 
\item $X_1 \cap X_2 \cap X_3 \cong \#_{g} S^1 \times S^1 = \Sigma_g$.
\end{itemize}
\end{dfn}

Let $H_{\alpha} = X_3 \cap X_1$, $H_{\beta} = X_1 \cap X_2$ and $H_{\gamma} = X_2 \cap X_3$. The union $H_{\alpha} \cup H_{\beta} \cup H_{\gamma}$ is called the \textit{spine} of the trisection. A trisection is uniquely determined by its spine \cite{MR3590351}. The 4-tuple $(g;k_1,k_2,k_3)$ is called the \textit{type} of the trisection. Note that if $k_1=k_2=k_3$, the trisection is said to be \textit{balanced}, and the type is denoted by $(g,k)$ for short, where $k=k_i$ ($i=1,2,3$). A $(g,k)$-trisection of a closed 4-manifold $X$ is also called a \textit{genus-g trisection} since $\chi(X)=2+g-3k$, where $\chi(X)$ is the Euler characteristic of $X$. Each $X_i$ is called a \textit{sector}.

\begin{dfn}\label{def:trisection diagram}
A 4-tuple $(\Sigma_g;\alpha,\beta,\gamma)$ is called a $(g;k_1,k_2,k_3)$-\textit{trisection diagram} if the following holds:
\begin{itemize}
\item $(\Sigma_g;\alpha,\beta)$ is a Heegaard diagram of $\#_{k_1}S^1 \times S^2$,
\item $(\Sigma_g;\beta,\gamma)$ is a Heegaard diagram of $\#_{k_2}S^1 \times S^2$ and 
\item $(\Sigma_g;\gamma,\alpha)$ is a Heegaard diagram of $\#_{k_3}S^1 \times S^2$.
\end{itemize}
\end{dfn}

Similar to the type of a trisection, the 4-tuple of non-negative integers $(g;k_1,k_2,k_3)$ is called the \textit{type} of the trisection diagram. The curves $\alpha$, $\beta$ and $\gamma$ are drawn in red, blue and green, respectively.

\begin{exm}
The left (resp. right) figure in Figure \ref{fig:CP^2} is a $(1,0)$-trisection diagram of $\mathbb{C}P^2$ (resp. $\overline{\mathbb{C}P^2}$). Figure \ref{fig:CP^2s} is the standard $(k+\ell,0)$-trisection diagram of $k \mathbb{C}P^2 \# \ell \overline{\mathbb{C}P^2}$.
\end{exm}

\begin{figure}[htbp]
  \centering
  \begin{minipage}{0.45\textwidth}
    \centering
    \includegraphics[width=\textwidth]{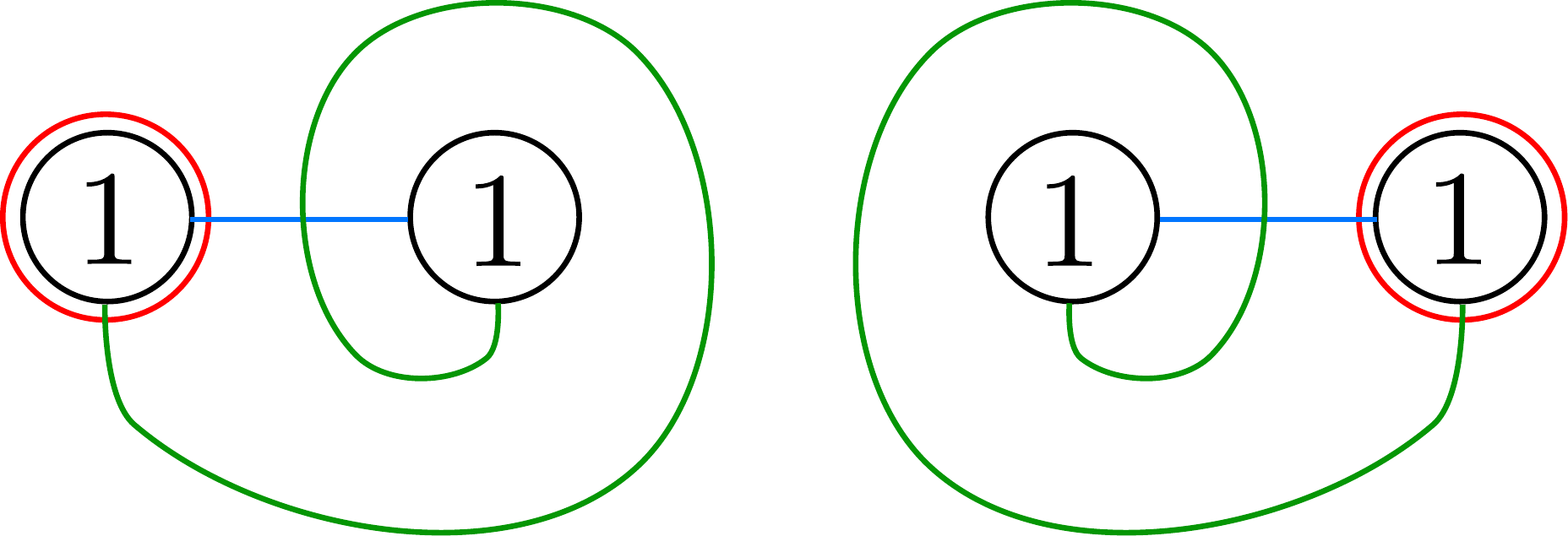}
    \setlength{\captionmargin}{2pt}
    \caption{(Left) A $(1,0)$-trisection diagram of $\mathbb{C}P^2$. (Right) A $(1,0)$-trisection diagram of $\overline{\mathbb{C}P^2}$.}
    \label{fig:CP^2}
  \end{minipage}
  \hfill
  \begin{minipage}{0.45\textwidth}
    \centering
    \includegraphics[width=\textwidth]{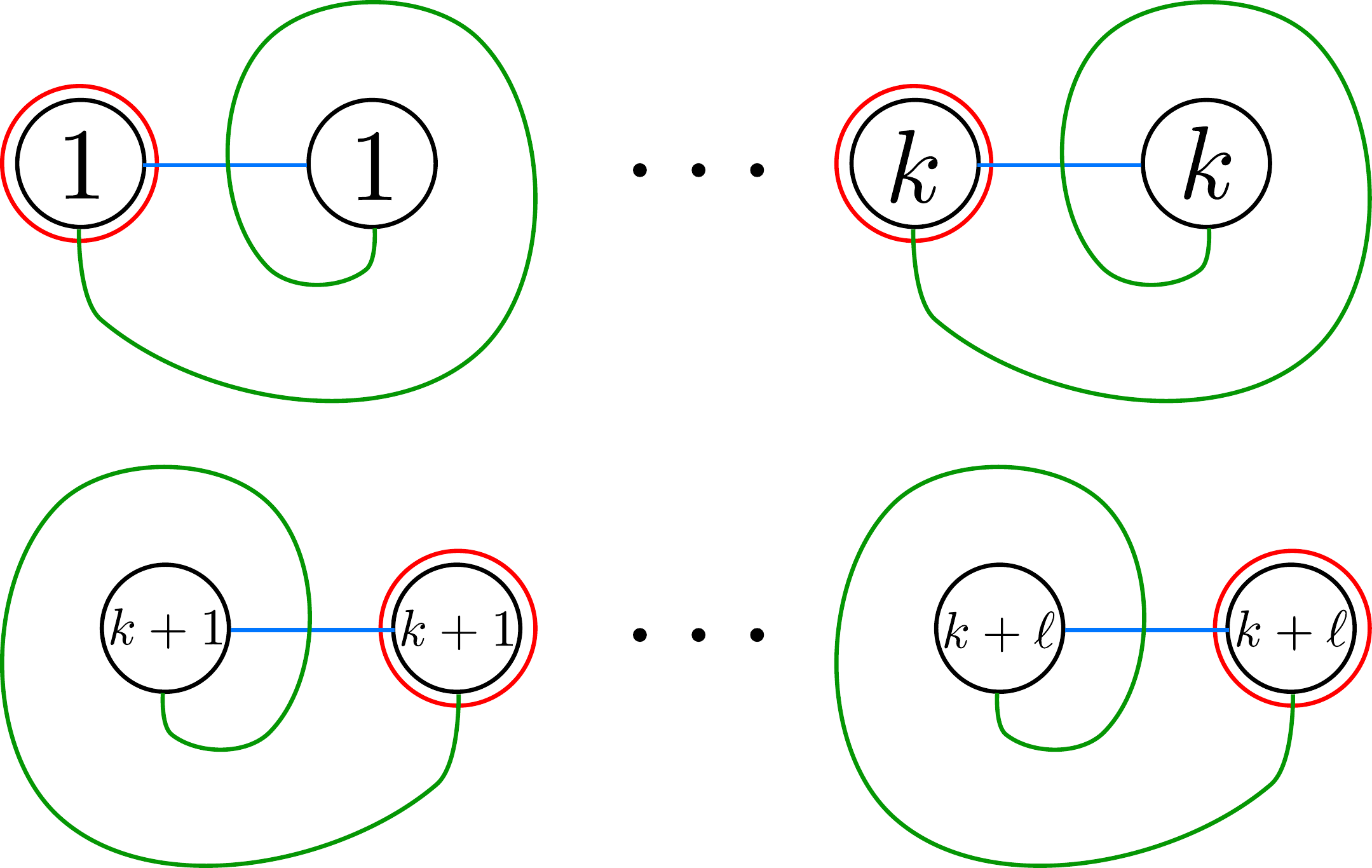}
    \setlength{\captionmargin}{5pt}
    \caption{The standard $(k+\ell,0)$-trisection diagram of $k \mathbb{C}P^2 \# \ell \overline{\mathbb{C}P^2}$.}
    \label{fig:CP^2s}
  \end{minipage}
\end{figure}

\begin{rem}\label{rem:correspond}
Given a genus-$g$ trisection whose spine is $H_{\alpha} \cup H_{\beta} \cup H_{\gamma}$, let $\alpha$ (resp. $\beta$, $\gamma$) be the boundary of meridian disk systems of $H_\alpha$ (resp. $H_\beta$, $H_\gamma$). Then, $(\Sigma_g;\alpha,\beta,\gamma)$ is the trisection diagram with respect to the trisection. Conversely, given a trisection diagram $(\Sigma_g;\alpha,\beta,\gamma)$, by attaching 2-handles to $\Sigma_g \times D^2$ along $\alpha \times \{e^{\frac{2{\pi}i}{3}}\}$, $\beta \times \{e^{\frac{4{\pi}i}{3}}\}$ and $\gamma \times \{e^{2{\pi}i}\}$ with surface framing, we can construct the trisected 4-manifold corresponding to the trisection diagram. Note that in the last process, we have only constructed the spine of the trisection since a trisection is uniquely determined by its spine.
\end{rem}

\begin{dfn}\label{def:stabili}
Let $(X_1,X_2,X_3)$ be a trisection and $C_{ij}$ a boundary-parallel arc properly embedded in $X_i \cap X_j$. For $\{i,j,k\}=\{1,2,3\}$, we define $X_i^{'}$, $X_j^{'}$ and $X_k^{'}$ as follows:
\begin{itemize}
\item $X_i^{'}=X_i-\nu(C_{ij})$,
\item $X_j^{'}=X_j-\nu(C_{ij})$ and
\item $X_k^{'}=X_k \cup \overline{\nu(C_{ij})}$,
\end{itemize}
where $\nu(C_{ij})$ is a tubular neighborhood of $C_{ij}$. 
The replacement of $(X_1,X_2,X_3)$ by $(X_1^{'},X_2^{'},X_3^{'})$ above is called the $k$-\textit{stabilization}. The reverse operation is called the \textit{k-destabilization}.
\end{dfn}

\begin{dfn}
A \textit{stabilization} of a trisection diagram is the connected-sum of the trisection diagram with the genus-1 trisection diagram of $S^4$ shown in Figure \ref{fig:stabilizationfordiagram}.
The reverse operation is called a \textit{destabilization}.
\end{dfn}

\begin{figure}[h]
\begin{center}
\includegraphics[width=8cm, height=3cm, keepaspectratio, scale=1]{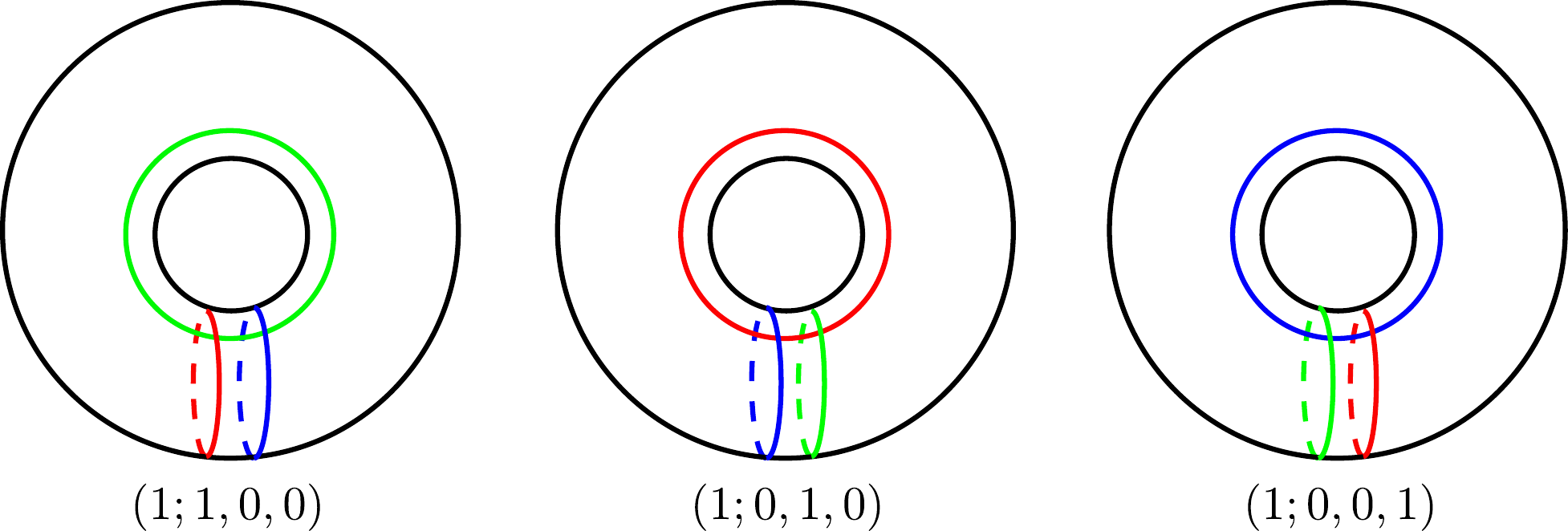}
\end{center}
\setlength{\captionmargin}{50pt}
\caption{Genus-1 trisection diagrams of $S^4$. From left to right, these correspond to the 1-, 2-, and 3-stabilizations.}
\label{fig:stabilizationfordiagram}
\end{figure}

Note that this (de)stabilization is the unbalanced one. Taking the connected sum with the leftmost (resp. center, rightmost) figure in Figure \ref{fig:stabilizationfordiagram} corresponds to the 1- (resp. 2-, 3-) stabilization in Definition \ref{def:stabili}.

\begin{dfn}
For diffeomorphic closed 4-manifolds $X$ and $Y$, two trisections $(X_1, X_2, X_3)$ of $X$ and $(Y_1,Y_2,Y_3)$ of $Y$ are \textit{diffeomorphic} if there exists a diffeomorphism $h \colon X \to Y$ such that $h(X_i)=Y_i$ for each $i=1,2,3$. 
Two trisections $(X_1, X_2, X_3)$ and $(Y_1,Y_2,Y_3)$ of the same closed 4-manifold $Z$ are \textit{isotopic} if there exists an isotopy $\{h_t \colon Z \to Z\}_{t \in [0,1]}$ such that $h_0=\operatorname{id}_{Z}$ and $h_1(X_i)=Y_i$ for each $i=1,2,3$. 
\end{dfn}

\begin{thm}[{\cite[Theorem 4, Theorem 11]{MR3590351}}]
Every closed 4-manifold admits a trisection. Any two trisections of the same closed 4-manifold are isotopic after performing some number of stabilizations.
\end{thm}

\begin{cor}[{\cite[Corollary 12]{MR3590351}}]
Two closed 4-manifolds are diffeomorphic if and only if two corresponding trisection diagrams are related by surface diffeomorphisms, handle slides among the same family curves and stabilizations.
\end{cor}

Two trisections are diffeomorphic if and only if two corresponding trisection diagrams are related without stabilizations. 

\begin{con}[{\cite[Conjecture 3.11]{MR3544545}}]\label{conj:4DWC}
Each trisection of $S^4$ is isotopic to the genus-0 trisection or its stabilization.
\end{con}

This conjecture is a 4-dimensional analogue of Waldhausen's theorem which states that each genus-$g$ Heegaard splitting of $S^3$ is isotopic to the stabilization of the genus-0 Heegaard splitting. A counterexample of this conjecture may be given via trisection diagrams of a 4-manifold diffeomorphic to $S^4$ since any two isotopic trisections are diffeomorphic.

Trisections of 4-manifolds with boundary are called relative trisections \cite{castro2016relative}. Let us recall them here.

For the disk $D=\{re^{i\theta} \mid r\in [0, 1] , -\pi/3\leq \theta \leq \pi / 3 \} \subset \mathbb{C}$,  we define $\partial ^{\pm} D \subset \partial{D}$ and $\partial^{0} D \subset \partial{D}$ as follows:
\[
	 \partial ^{-} D = \{re^{\pi/3} \in \partial D \mid r\in [0, 1]\}, 
\] 
\[
	\partial^{0} D = \{e^{i\theta}\in \partial D\} \ \text{and}
\]
\[
	 \partial^{+} D = \{re^{-\pi/3} \in \partial D \mid r\in [0, 1]\}.
\]
Then, we have $\partial{D}=\partial ^{-} D \cup \partial^{0} D \cup  \partial^{+} D$.
For a genus-$p$ surface $\Sigma_{p,b}$ with $b$ boundary components, $U:=\Sigma_{p,b} \times D$ is diffeomorphic to the boundary connected-sum $\natural_{2p+b-1} S^1\times B^3$ since $\Sigma_{p,b}$ consists of a 0-handle and $2p+b-1$ 1-handles.
Furthermore, its boundary is decomposed into $\partial ^0 U = (\Sigma_{p,b} \times \partial ^{0} D) \cup (\partial \Sigma_{p,b} \times D)$ and into $\partial ^{\pm} U = \Sigma_{p,b} \times \partial^{\pm} D$.
Let $V_n$ be $\natural_n (S^1\times B^3)$ and  $\partial V_n = H_s^{-}\cup H_{s}^{+}$ a genus-$(n+s)$ Heegaard splitting of $\partial V_n$ for non-negative integers $n$ and $s$.

Let $s=g-k+p+b-1$ and $n=k-2p-b+1$. Then, we have $Z_k\cong U\natural V_n\cong \natural_k S^1\times B^3$.
Note that 
\[
	\partial Z_k = Y^{+}_{g, k ; p, b} \cup \partial ^0 U \cup Y^{-}_{g, k ; p, b}
\]
, where $Y^{\pm}_{g, k ; p, b}=\partial ^{\pm} U\natural H_{s}^{\pm}$.
Under the above settings, relative trisections are defined as follows.

\begin{dfn}
	Let $W$ be a 4-manifold with connected boundary. A decomposition $W=W_1\cup W_2\cup W_3$ of $W$ is called a $(g, k; p, b)$-{\it relative trisection} of W if the following conditions hold: 
	\begin{itemize}
		\item $W_i\cong Z_k$ for $i=1, 2, 3$ and
		\item $W_i\cap W_{i+1} \cong Y^{+}_{g, k ; p, b}$ and $W_i\cap W_{i-1}\cong Y^{-}_{g, k ; p, b}$ for $i=1, 2, 3$.
	\end{itemize}
\end{dfn}

It follows from these conditions that $W_1\cap W_2\cap W_3$ is a genus-$g$ surface $\Sigma_{g,b}$ with $b$ boundary components. Note that relative trisections in this definition are balanced. We can define unbalanced relative trisections similar to the closed case.

\begin{lem}[{\cite[Lemma 11]{MR3770114}}]\label{lem:open book decomposition}
A $(g,k;p,b)$-relative trisection of a 4-manifold $X$ with non-empty boundary induces an open book decomposition on $\partial{X}$ with page $\Sigma_{p,b}$. 
\end{lem}

As in the closed case, relative trisection diagrams can be defined for relative trisections \cite{MR3770114}.

\begin{dfn}
For a genus-$g$ surface $\Sigma$ with $b$ boundary components and three sets of $g-p$ simple closed curves $\alpha$, $\beta$ and $\gamma$ on $\Sigma$, a $(g, k; p, b)$-{\it relative trisection diagram} is a 4-tuple $(\Sigma; \alpha, \beta, \gamma)$ such that each of $(\Sigma; \alpha, \beta)$, $(\Sigma; \beta, \gamma)$ and $(\Sigma; \gamma,\alpha)$ is related to the diagram shown in Figure \ref{reltridiag} by surface diffeomorphisms and handle slides among the same family curves.
\end{dfn}

\begin{figure}[h]
\centering
\includegraphics[scale=0.6]{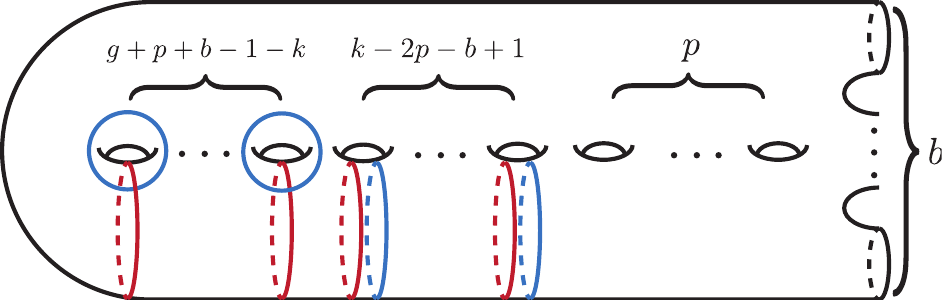}
\caption
{The standard diagram for relative trisection diagrams.}
\label{reltridiag}
\end{figure}

We can use the following lemma if we glue two relative trisections.

\begin{lem}[{\cite[Lemma 2.7]{MR3999550}}]\label{lem:gluing}
Let $X$ and $X'$ be 4-manifolds with non-empty and connected boundary, $T$ and $T'$ relative trisections of $X$ and $X'$, respectively, and $\mathcal{O}X$ and $\mathcal{O}X'$ the open book decompositions on $\partial{X}$ and $\partial{X}'$ induced by $T$ and $T'$, respectively. If $f \colon \partial{X} \to \partial{X}'$ is an orientation reversing diffeomorphism that takes $\mathcal{O}X$ to $\mathcal{O}X'$, then we obtain a trisection of the closed 4-manifold $\hat X=X \cup_{f} X'$ by gluing $T$ and $T'$.
\end{lem}

Note that if there exists a diffeomorphism $f$ as above, the page of $\mathcal{O}X$ is diffeomorphic to that of $\mathcal{O}X'$ via $f$. Thus, if $T$ and $T'$ are $(g,k;p,b)$- and $(g',k';p',b')$-relative trisections, then $p=p'$ and $b=b'$.

If we glue two relative trisection diagrams, we must construct relative trisection diagrams with arcs, which are called arced relative trisection diagrams. 

\begin{dfn}\label{def:artd}
Let $(\Sigma;\alpha,\beta,\gamma)$ be a $(g,k;p,b)$-relative trisection diagram, and  $\mathfrak{a}$, $\mathfrak{b}$ and $\mathfrak{c}$ collections
of $2p$ properly embedded arcs in $\Sigma$. Suppose that the arcs $\mathfrak{a}$ (resp. $\mathfrak{b}$ and $\mathfrak{c}$) are disjoint from $\alpha$ (resp. $\beta$ and $\gamma$). A tuple $(\Sigma;\alpha,\beta,\gamma,\mathfrak{a},\mathfrak{b},\mathfrak{c})$ is called \textit{an arced relative trisection diagram} if the following conditions hold:
\begin{itemize}

\item Cutting $\Sigma_\alpha$ (resp. $\Sigma_\beta$, $\Sigma_\gamma$) along $\mathfrak{a}$ (resp. $\mathfrak{b}$, $\mathfrak{c}$) yields a disk, where $\Sigma_x$ is a surface obtained by surgering $\Sigma$ along $x$ for $x \in \{\alpha, \beta, \gamma\}$.

\item By handle slides, $(\Sigma;\alpha,\beta,\mathfrak{a},\mathfrak{b})$ is related to some $(\Sigma;\alpha{'},\beta{'},\mathfrak{a}{'},\mathfrak{b}{'})$ such that $(\Sigma;\alpha{'},\beta{'})$ is related to the diagram shown in Figure \ref{reltridiag} and $\mathfrak{a}{'} = \mathfrak{b}{'}$.

\item By handle slides, $(\Sigma;\beta,\gamma,\mathfrak{b},\mathfrak{c})$ is related to some $(\Sigma;\beta{'},\gamma{'},\mathfrak{b}{'},\mathfrak{c}{'})$ such that $(\Sigma;\beta{'},\gamma{'})$ is related to the diagram shown in Figure \ref{reltridiag} and $\mathfrak{b}{'} = \mathfrak{c}{'}$.
\end{itemize}
\end{dfn}

An algorithm for drawing the arcs is developed in \cite[Theorem 5]{MR3770114}.

\begin{thm}[{\cite[Theorem 5]{MR3770114}}]\label{thm:algorithm}
Let $(\Sigma;\alpha,\beta,\gamma)$ be a relative trisection diagram, $\Sigma_\alpha$ the surface obtained by surgering $\Sigma$ along $\alpha$ and $\phi \colon \Sigma - \alpha \to \Sigma_\alpha$ an associated embedding. Then, collections $\mathfrak{a}$, $\mathfrak{b}$ and $\mathfrak{c}$ of arcs in $\Sigma$ are obtained by taking the procedure below in order.
\begin{enumerate}
\item There exists a collection of properly embedded arcs $\delta$ in $\Sigma_\alpha$ such that cutting $\Sigma_\alpha$ along $\delta$ yields a disk. Then, let $\mathfrak{a}$ denote a collection of properly embedded arcs in $\Sigma-\alpha$ such that $\delta$ is isotopic to $\phi(\mathfrak{a})$ in $\Sigma_\alpha$.
\item A collection of arcs in $\Sigma$ that is disjoint from $\beta$ can be obtained by sliding a copy of $\mathfrak{a}$ over $\alpha$. Then, let $\mathfrak{b}$ denote the collection of arcs. Note that in this operation, sliding a $\beta$ curve over other $\beta$ curves can be performed if necessary.
\item A collection of arcs in $\Sigma$ that is disjoint from $\gamma$ can be obtained by sliding a copy of $\mathfrak{b}$ over $\beta$. Then, let $\mathfrak{c}$ denote the collection of arcs. Note that in this operation, sliding a $\gamma$ curve over other $\gamma$ curves can be performed if necessary.
\end{enumerate}
\end{thm}

Let $(\Sigma;\alpha,\beta,\gamma)$ and $(\Sigma';\alpha',\beta',\gamma')$ be $(g,k;p,b)$- and $(g',k';p',b')$-relative trisection diagrams of 4-manifolds $X$ and $X'$, respectively, and $(\mathfrak{a},\mathfrak{b},\mathfrak{c})$ a collection of arcs for $(\Sigma;\alpha,\beta,\gamma)$ in Definition \ref{def:artd}, where $\mathfrak{a}=\{a_1,a_2, \ldots, a_\ell\}$ and $\ell=2p+b-1$. If there exists $f$ in Lemma \ref{lem:gluing}, we can obtain a collection of arcs $\mathfrak{a}'=\{a_1',a_2', \ldots, a_\ell'\}$ that cuts $\Sigma_{\alpha'}$ into a disk, where $a_i'=f(a_i)$ for $i=1,2,\ldots \ell$. Then, by Theorem \ref{thm:algorithm}, we have a collection of arcs $(\mathfrak{a}',\mathfrak{b}',\mathfrak{c}')$ for $(\Sigma';\alpha',\beta',\gamma')$. After that, we can obtain three kinds of new simple closed curves in $\Sigma \cup_{f} \Sigma'$, i.e. $\mathfrak{a} \cup_\partial \mathfrak{a}'$, $\mathfrak{b} \cup_\partial \mathfrak{b}'$ and $\mathfrak{c} \cup_\partial \mathfrak{c}'$ via $f$. Then, we have the following proposition, where $\hat\Sigma=\Sigma \cup_{f} \Sigma'$ and $\tilde{\alpha}$ (resp. $\tilde{\beta}$, $\tilde{\gamma}$) $=(\mathfrak{a}_j \cup_{\partial} \mathfrak{a}'_j)_j$ (resp. $=(\mathfrak{b}_j \cup_{\partial} \mathfrak{b}'_j)_j$, $=(\mathfrak{c}_j \cup_{\partial} \mathfrak{c}'_j)_j$).

\begin{prop}[{\cite[Proposition 2.12]{MR3999550}}]
Under the assumptions of Lemma \ref{lem:gluing} and the above setup, the 4-tuple $(\hat\Sigma, \hat\alpha, \hat\beta, \hat\gamma)$ is a trisection diagram of the closed 4-manifold $\hat X=X \cup_{f} X'$, where $\hat\alpha=\alpha \cup \alpha' \cup \tilde{\alpha}$, $\hat\beta=\beta \cup \beta' \cup \tilde{\beta}$ and $\hat\gamma=\gamma \cup \gamma' \cup \tilde{\gamma}$.
\end{prop}

\section{Main Theorem}\label{sec:main theorem}

We first construct a Lefschetz fibration of the complement of a neighborhood of a regular fiber.

\begin{lem}\label{lem:complement}
Let $X$ be a closed 4-manifold admitting a genus-$p$ achiral Lefschetz fibration over $S^2$ with monodromy $\varphi \in \operatorname{Mod}(\Sigma_{p})$. Then, $X-(\Sigma_p \times D^2)$ admits an achiral Lefschetz fibration over $D^2$ with regular fiber a genus-$p$ surface $\Sigma_{p,2}$ with two boundary components and monodromy $\varphi \circ t_{\delta_1}^{\pm1} \circ t_{\delta_2}^{\mp1} \in \operatorname{Mod}(\Sigma_{p,2})$, where $\delta_1$ and $\delta_2$ are the boundary parallel curves depicted in Figure \ref{scc1}.
\end{lem}

\begin{figure}[h]
\begin{center}
\includegraphics[width=5cm,page=2]{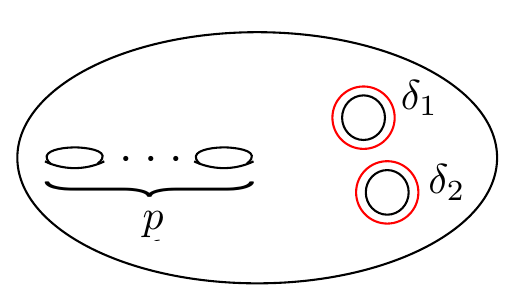}
\end{center}
\setlength{\captionmargin}{50pt}
\caption{The boundary parallel curves $\delta_1, \delta_2$ on $\Sigma_{p,2}$.}
\label{scc1}
\end{figure}

\begin{figure}[h]
\begin{center}
\includegraphics[width=8cm,page=9]{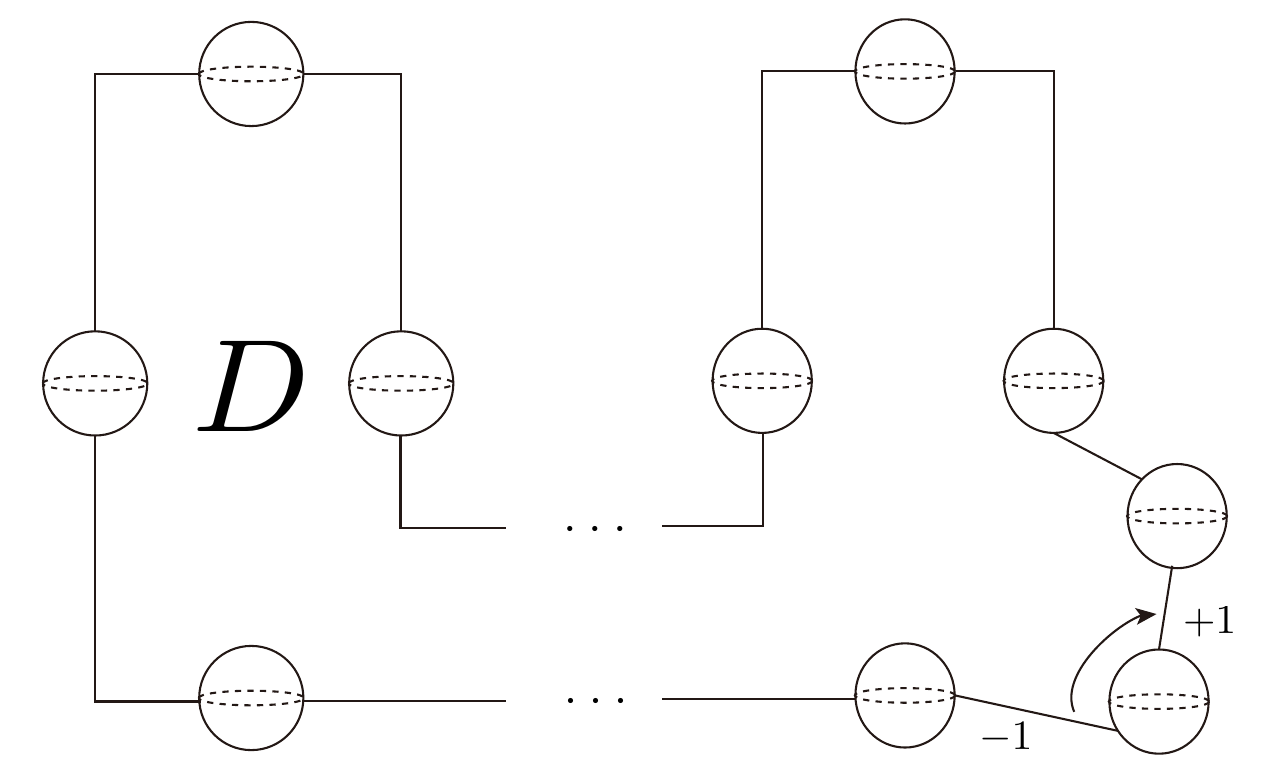}
\end{center}
\setlength{\captionmargin}{50pt}
\caption{The handle diagram arising from the achiral Lefschetz fibration over $D^2$ with $\varphi \circ t_{\delta_1}^{+1} \circ t_{\delta_2}^{-1}$.} 
\label{HD1}
\end{figure}

\begin{figure}[h]
\begin{center}
\includegraphics[width=8cm,page=10]{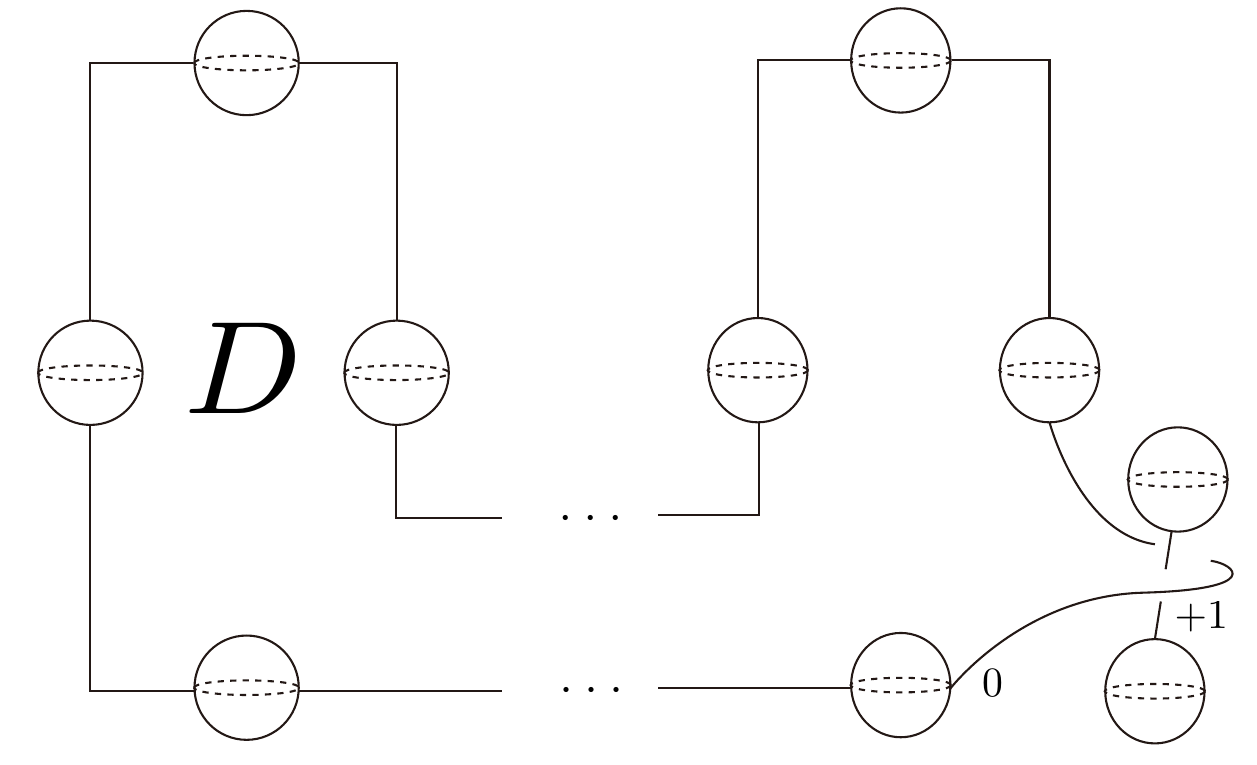}
\end{center}
\setlength{\captionmargin}{50pt}
\caption{The handle diagram obtained by handle sliding.} 
\label{HD2}
\end{figure}

\begin{figure}[h]
\begin{center}
\includegraphics[width=8cm,page=11]{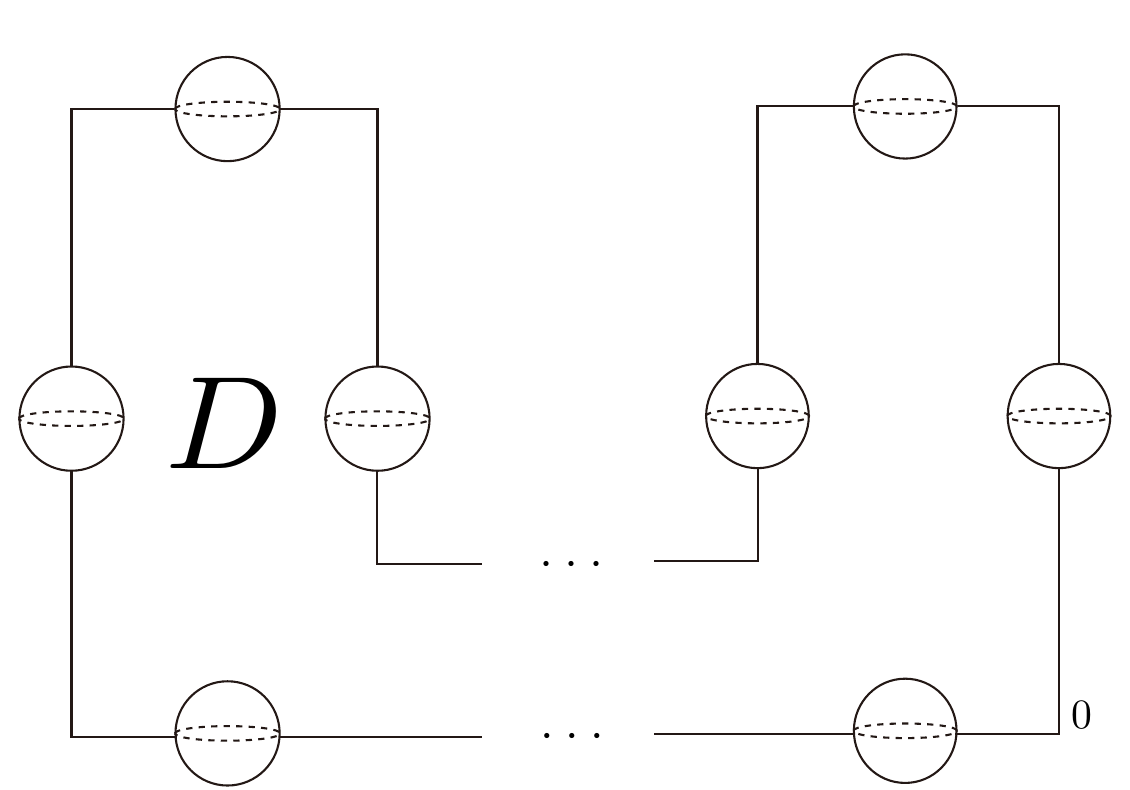}
\end{center}
\setlength{\captionmargin}{50pt}
\caption{The handle diagram of $X - (\Sigma_p \times D^2)$.} 
\label{HD3}
\end{figure}

\begin{proof}
Let $D$ in Figure~\ref{HD1} be the 2–handles corresponding to the vanishing cycles of the monodromy $\varphi\in \operatorname{Mod}(\Sigma_{p})$ of the achiral Lefschetz fibration over $S^2$ in the handle diagram of $X$.  
In the case of $\varphi \circ t_{\delta_1}^{+1}\circ t_{\delta_2}^{-1}\in \operatorname{Mod}(\Sigma_{p,2})$, Figure~\ref{HD1} represents the whole handle diagram arising from the achiral Lefschetz fibration over $D^2$.
This figure can be drawn from the method in Subsection \ref{subsec:handle diagram}.
Then, we obtain Figure~\ref{HD2} by performing the handle slide as indicated by the arrow and canceling the 1-, 2-handle pair.  
This Figure~\ref{HD3} clearly gives the handle diagram of $X - (\Sigma_p \times D^2)$.
Since the same argument applies to the case with opposite signs, the statement holds.
\end{proof}

Let $n \mathbb{C}P^2$ denote $n \mathbb{C}P^2$ for a positive integer $n$, $(-n) \overline{\mathbb{C}P^2}$ 
for a negative integer $n$ and $S^4$ for $n=0$.
We next construct a Lefschetz fibration of the connected sum of a neighborhood of a regular fiber with $n \mathbb{C}P^2$.

\begin{lem}\label{lem:fiber}
For a closed orientable surface $\Sigma_p$ of genus-$p$, the 4-manifold $\Sigma_p \times D^2 \# n \mathbb{C}P^2$ admits an achiral Lefschetz fibration over $D^2$ with regular fiber a genus-$p$ surface $\Sigma_{p,2}$ with two boundary components and monodromy $t_{\delta_3}^{-n} \circ t_{\delta_1}^{\mp1} \circ t_{\delta_2}^{\pm1}\in \operatorname{Mod}(\Sigma_{p,2})$, where $\delta_1$, $\delta_2$ and $\delta_3$ are the curves depicted in Figure \ref{scc2}.
\end{lem}

\begin{figure}[h]
\begin{center}
\includegraphics[width=5cm,page=3]{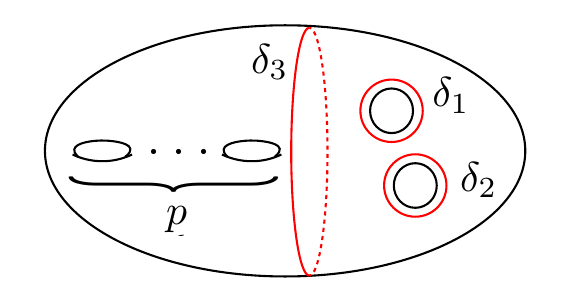}
\end{center}
\setlength{\captionmargin}{50pt}
\caption{The simple closed curves $\delta_1, \delta_2, \delta_3$ on $\Sigma_{p,2}$.}
\label{scc2}
\end{figure}

\begin{figure}[h]
\begin{center}
\includegraphics[width=8cm,page=5]{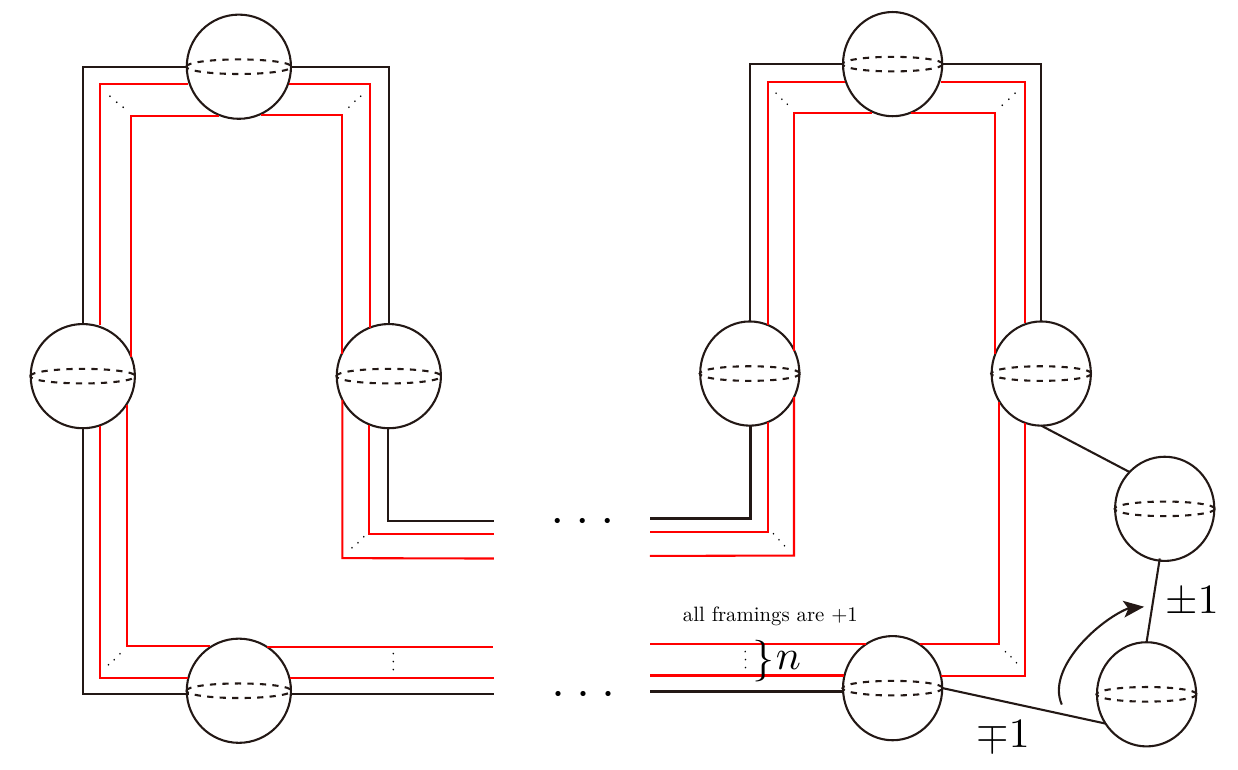}
\end{center}
\setlength{\captionmargin}{50pt}
\caption{The handle diagram arising from the achiral Lefschetz fibration over $D^2$ with $t_{\delta_3}^{-n} \circ t_{\delta_1}^{\mp1} \circ t_{\delta_2}^{\pm1}\in \operatorname{Mod}(\Sigma_{p,2})$.} 
\label{HD4}
\end{figure}

\begin{figure}[h]
\begin{center}
\includegraphics[width=8cm,page=6]{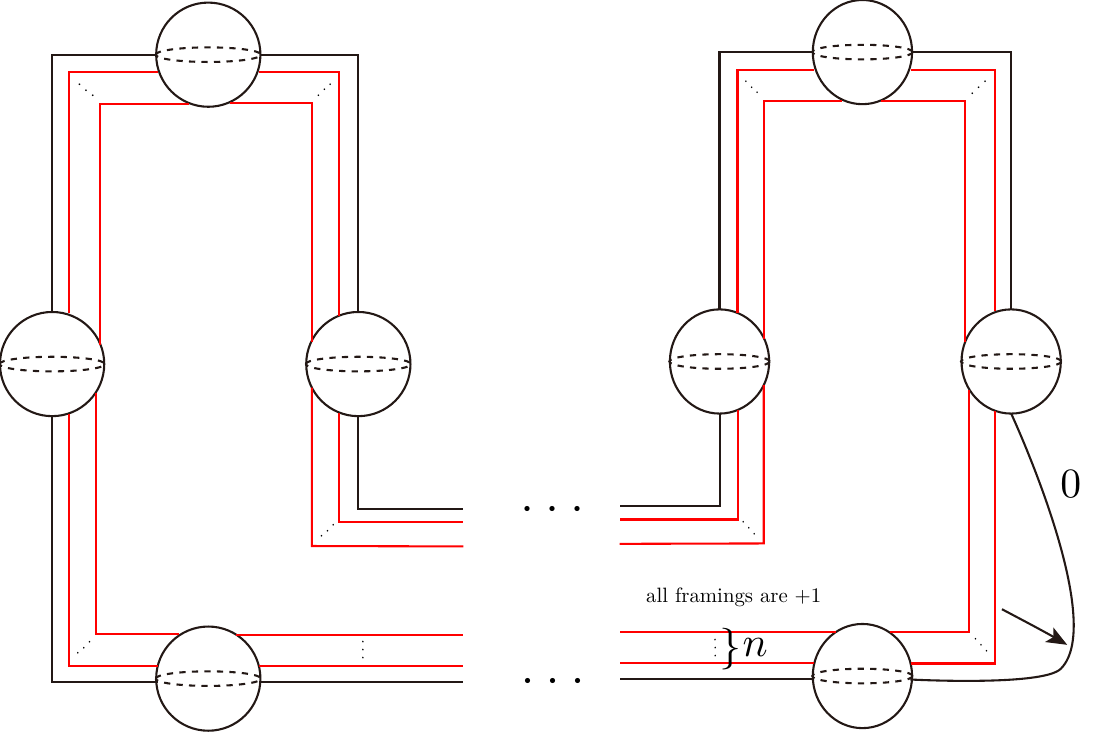}
\end{center}
\setlength{\captionmargin}{50pt}
\caption{The handle diagram obtained by handle sliding and canceling.} 
\label{HD5}
\end{figure}

\begin{figure}[h]
\begin{center}
\includegraphics[width=8cm,page=7]{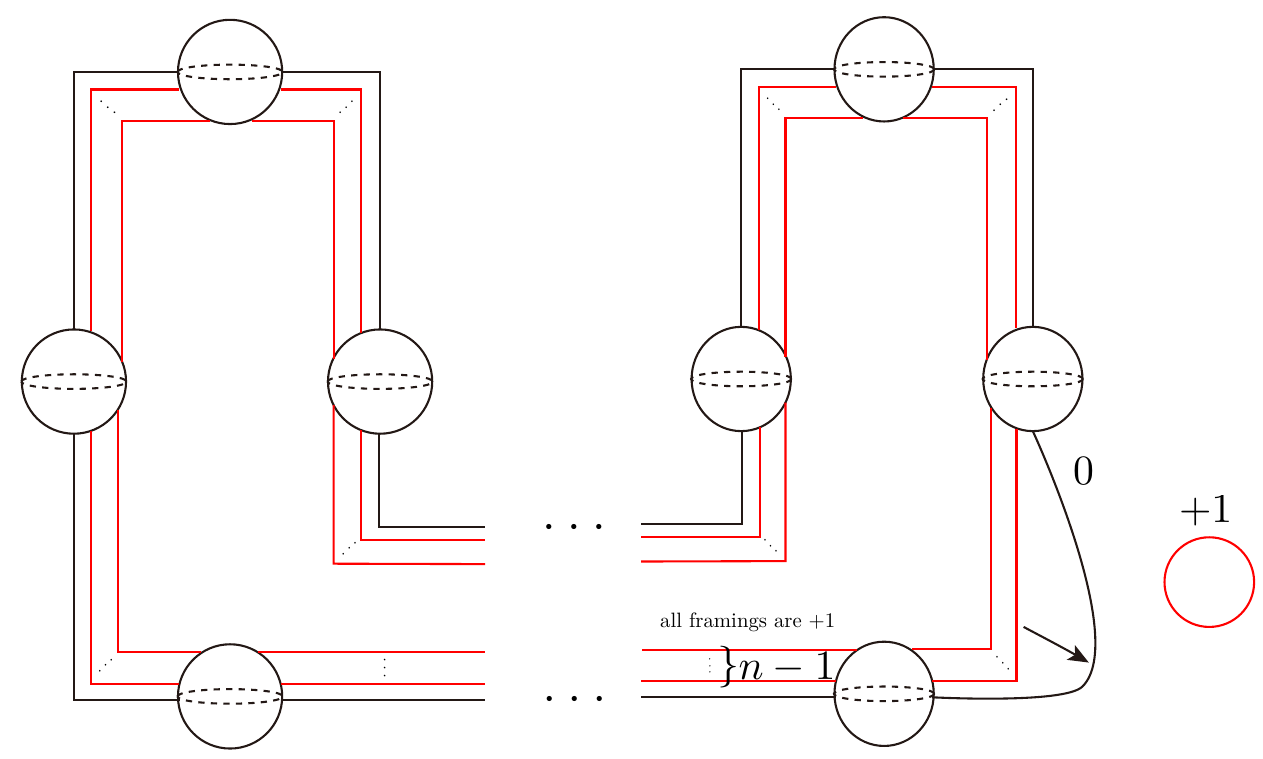}
\end{center}
\setlength{\captionmargin}{50pt}
\caption{The handle diagram obtained by handle sliding.} 
\label{HD6}
\end{figure}

\begin{figure}[h]
\begin{center}
\includegraphics[width=8cm,page=8]{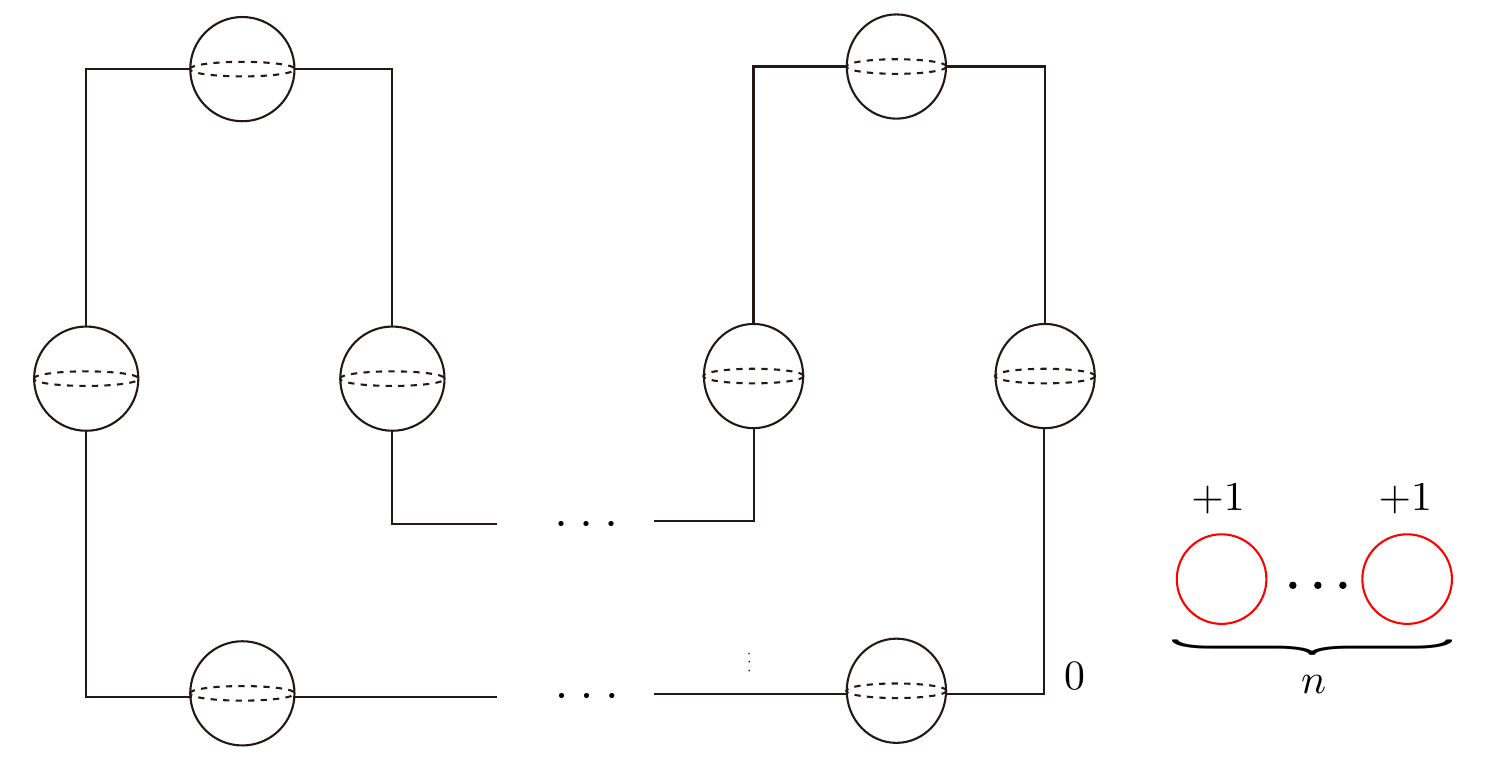}
\end{center}
\setlength{\captionmargin}{50pt}
\caption{The handle diagram of $\Sigma_p \times D^2 \#{n} \mathbb{C}P^2$.} 
\label{HD7}
\end{figure}

\begin{proof}
We obtain a handle diagram shown in Figure \ref{HD4} from the monodromy $t_{\delta_3}^{-n} \circ t_{\delta_1}^{\mp1} \circ t_{\delta_2}^{\pm1}\in \operatorname{Mod}(\Sigma_{p,2})$ in the case where $n>0$.
Here, all the red 2-handles have (+1)-framing.
As in Lemma~\ref{lem:complement}, Figure~\ref{HD5} represents the handle diagram obtained by performing the handle slide indicated by the arrow and canceling the 1-, 2-handle pair.
Then, by sliding the 2–handle corresponding to $\delta_3$ over the 0–framed 2–handles, we obtain Figure~\ref{HD6}.
Repeating this process $(n-1)$ times yields Figure~\ref{HD7}, which is a handle diagram of $\Sigma_p \times D^2 \# n \mathbb{C}P^2$. Since the same argument applies to the case where $n<0$, the  statement holds.
\end{proof} 

The following lemma says the way of converting Lefschetz fibrations into relative trisections.

\begin{lem}[{\cite[Corollary 17]{MR3770114}, \cite[Lemma 3.4]{MR3999550}}]\label{lem:LF_rtd}
Let $X$ be a 4-manifold admitting an achiral Lefschetz fibration over $D^2$ with regular fiber a genus-$p$ surface $\Sigma_{p,b}$ with b boundary components and $n$ singular fibers. Then, $X$ admits a $(p+n,2p+b-1;p,b)$-relative trisection whose induced open book decomposition is the same as that canonically induced from the Lefschetz fibration. Furthermore, the corresponding relative trisection diagram can be constructed explicitly from a monodromy of the Lefschetz fibration.
\end{lem}

Figure \ref{fig:convert} describes a method for obtaining the relative trisection diagram in Lemma \ref{lem:LF_rtd}. See \cite[Figure 3]{MR3999550} for details.

\begin{figure}[h]
\begin{center}
\includegraphics[width=8cm, height=8cm, keepaspectratio, scale=1]{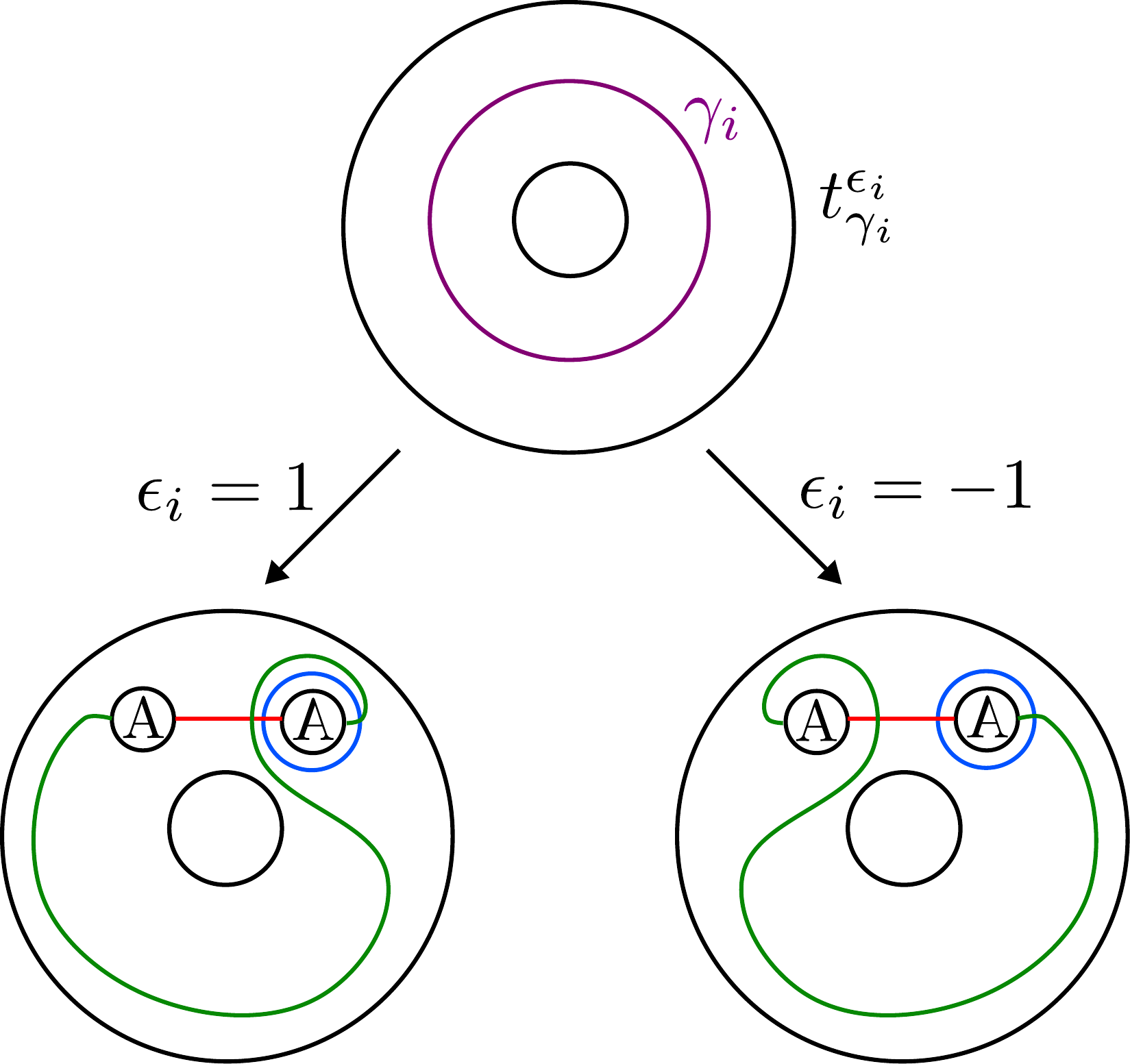}
\end{center}
\setlength{\captionmargin}{50pt}
\caption{A method for converting a vanishing cycle $\gamma_i$ corresponding to the Dehn twist $t_{\gamma_I}^{\epsilon_i}$ that is a part of a monodromy into a relative trisection. The top figure describes an annular neighborhood of $\gamma_i$.}
\label{fig:convert}
\end{figure}

The following theorem is the first main theorem.

\begin{thm}\label{thm:-n-section}
Let $X$ be a closed 4-manifold admitting a genus-$p$ achiral Lefschetz fibration over $S^2$ with $m$ singular fibers and a $(-n)$-section, where $n$ is any integer. Then, $X \# n \mathbb{C}P^2$ admits a $(2p+m+\abs{n}+5, 2p+1)$-trisection whose corresponding trisection diagram can be constructed explicitly from a monodromy of the Lefschetz fibration of $X$.
\end{thm}

\begin{proof}
We decompose $X \# n \mathbb{C}P^2$ into $X-\Sigma_p \times D^2$ and $\Sigma_p \times D^2 \# n \mathbb{C}P^2$, i.e.
\[
X \# n \mathbb{C}P^2 = (X-\Sigma_p \times D^2) \cup_\partial (\Sigma_p \times D^2 \# n \mathbb{C}P^2).
\]
Let $\varphi$ be a monodromy of the achiral Lefschetz fibration of $X$.
From Lemmas \ref{lem:complement} and \ref{lem:LF_rtd}, $X-\Sigma_p \times D^2$ admits a $(p+m+2,2p+1;p,2)$-relative trisection whose monodromy of the induced open book decomposition is $\varphi \circ t_{\delta_1}^{+1} \circ t_{\delta_2}^{-1}\in \operatorname{Mod}(\Sigma_{p,2})$. Then, since the Lefschetz fibration has a $(-n)$-section, we have $\varphi = t_{\delta_3}^n\in \operatorname{Mod}(\Sigma_{p,1})$. Thus, the monodromy of the induced open book decomposition is $t_{\delta_3}^n \circ t_{\delta_1}^{+1} \circ t_{\delta_2}^{-1}\in \operatorname{Mod}(\Sigma_{p,2})$. Similar to this, from Lemmas \ref{lem:fiber} and \ref{lem:LF_rtd}, $\Sigma_p \times D^2 \# n \mathbb{C}P^2$ admits a $(p+\abs{n}+2,2p+1;p,2)$-relative trisection whose monodromy of the induced open book decomposition is $t_{\delta_3}^{-n} \circ t_{\delta_1}^{-1} \circ t_{\delta_2}^{+1}\in \operatorname{Mod}(\Sigma_{p,2})$. Thus, by Lemma \ref{lem:gluing}, we can glue these two relative trisections to obtain a $(2p+m+\abs{n}+5, 2p+1)$-trisection of $X \# n \mathbb{C}P^2$. This completes the proof.
\end{proof}

The following theorem is the second main theorem, which constructs partially a trisection of the fiber sum.

\begin{thm}\label{thm:fibersum}
Let $X_i$ be closed 4-manifolds admitting genus-$p$ achiral Lefschetz fibrations over $S^2$ with $n_i$ singular fibers for $i=1,2$. Suppose that the Lefschetz fibrations of $X_1$ and $X_2$ have $(-n)$- and $n$-sections, respectively, where $n$ is any positive integer. Then, the fiber sum $X_1 \#_F X_2$ admits a $(2p+n_1+n_2+5, 2p+1)$-trisection whose corresponding trisection diagram can be constructed explicitly from monodromies of the Lefschetz fibrations of $X_1$ and $X_2$.
\end{thm}

\begin{proof}
By the definition of the fiber sum, we can decompose the total space $X_1 \#_F X_2$ as follows:
\[
X_1 \#_F X_2 = (X_1-\Sigma_p \times D^2) \cup_\partial (X_2-\Sigma_p \times D^2).
\]
Similar to the proof of Theorem \ref{thm:-n-section}, from Lemmas \ref{lem:fiber} and \ref{lem:LF_rtd}, $X_1-\Sigma_p \times D^2$ (resp. $X_2-\Sigma_p \times D^2$) admits a $(p+n_1+2,2p+1;p,2)$- (resp. $(p+n_2+2,2p+1;p,2)$-) relative trisection whose monodromy of the induced open book decomposition is $t_{\delta_3}^n \circ t_{\delta_1}^{+1} \circ t_{\delta_2}^{-1}\in \operatorname{Mod}(\Sigma_{p,2})$ (resp. $t_{\delta_3}^{-n} \circ t_{\delta_1}^{-1} \circ t_{\delta_2}^{+1}\in \operatorname{Mod}(\Sigma_{p,2})$). Thus, by Lemma \ref{lem:gluing}, we can glue these two relative trisections to obtain a $(2p+n_1+n_2+5, 2p+1)$-trisection of $X_1 \#_F X_2$. This completes the proof.
\end{proof}

\begin{rem}
If a non-trivial chiral Lefschetz fibration over $S^2$ has an $n$-section, then $n \le -1$ \cite{McDuff1990,Stipsicz1999}.
\end{rem}

\begin{dfn}
Let $X$ be a simply-connected closed 4-manifold. We say that $X$ is {\it almost completely decomposable} if $X \# \mathbb{C}P^2$ is diffeomorphic to $k \mathbb{C}P^2 \# \ell \overline{\mathbb{C}P^2}$ for some non-negative integers $k$ and $\ell$.
\end{dfn}

\begin{cor}\label{cor:ACD}
Let $X$ be an almost completely decomposable 4-manifold admitting a genus-$p$ achiral Lefschetz fibration over $S^2$ with $m$ singular fibers and a $(-n)$-section, where $n$ is any positive integer. Then, for some non-negative integers $k$ and $\ell$, $k \mathbb{C}P^2 \# \ell \overline{\mathbb{C}P^2}$ admits a $(2p+m+n+5, 2p+1)$-trisection whose corresponding trisection diagram can be constructed explicitly from a monodromy of the Lefschetz fibration of $X$.
\end{cor}

\begin{proof}
From Theorem \ref{thm:-n-section}, we have a $(2p+m+n+5, 2p+1)$-trisection of $X \#n \mathbb{C}P^2$. Since $X$ is almost completely decomposable, $X \#n \mathbb{C}P^2$ is diffeomorphic to $k \mathbb{C}P^2 \# \ell \overline{\mathbb{C}P^2}$ for some non-negative integers $k$ and $\ell$. This completes the proof.
\end{proof}

\begin{rem}
It is known \cite{zbMATH03505878, zbMATH03608421, zbMATH03729115, zbMATH00016235, zbMATH07206908} that the elliptic surface $E(n)$, the logarithmic transformation $E(n)_{p,q}$ of $E(n)$ and the knot surgered 4-manifold $E(n)_K$ of $E(n)$ along a knot $K$ are almost completely decomposable for each $n \ge 1$.
\end{rem}

\begin{exm}\label{exm:E(n)}
For each positive integer $n$, the elliptic surface $E(n)$ admits a chiral Lefschetz fibration over $S^2$ with monodromy $(t_bt_a)^{6n}=\operatorname{id}\in \operatorname{Mod}(\Sigma_{1})$ and a $(-n)$-section (i.e. $(t_bt_a)^{6n}=t_{\delta}^n\in \operatorname{Mod}(\Sigma_{1,1})$), where the boundary component $\delta$ and the simple closed curves $a$ and $b$ are the ones shown in Figure \ref{scc0_3}. Then, from Theorem \ref{thm:-n-section}, we can construct a $(13n+7,3)$-trisection diagram of $E(n) \# n \mathbb{C}P^2$ as follows: By Lemma \ref{lem:LF_rtd}, we can construct a $(12n+3,3;1,2)$-relative trisection diagram of $E(n)-T^2 \times D^2$ as in Figure \ref{fig:E(n)-T2D2_rtd}. We can also construct a $(n+3,3;1,2)$-relative trisection diagram of $T^2 \times D^2 \# n \mathbb{C}P^2$ as in Figure \ref{fig:T2D2nCP2_rtd} by Lemma \ref{lem:LF_rtd}. Thus, by gluing these two relative trisection diagrams with arcs shown in Figures \ref{fig:E(n)-T2D2_artd} and \ref{fig:T2D2nCP2_artd}, we can obtain the $(13n+7,3)$-trisection diagram of $E(n) \# n \mathbb{C}P^2$ as in Figure \ref{fig:E(n)nCP2}. It is known that $E(n) \# n \mathbb{C}P^2$ is diffeomorphic to $k \mathbb{C}P^2 \# \ell \overline{\mathbb{C}P^2}$, where $k=3n-1$ and $\ell=10n-1$. 
\end{exm}

\begin{figure}[h]
\begin{center}
\includegraphics[width=5cm,page=1]{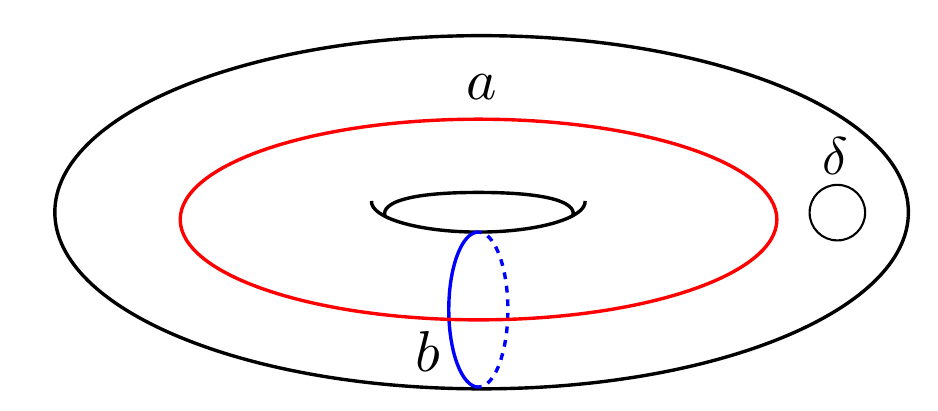}
\end{center}
\setlength{\captionmargin}{40pt}
\caption{The boundary $\delta$ and the simple closed curves $a,b$ on $\Sigma_{1,1}$.}
\label{scc0_3}
\end{figure}

\begin{figure}[htbp]
  \centering
  \begin{minipage}{0.45\textwidth}
    \centering
    \includegraphics[width=\textwidth]{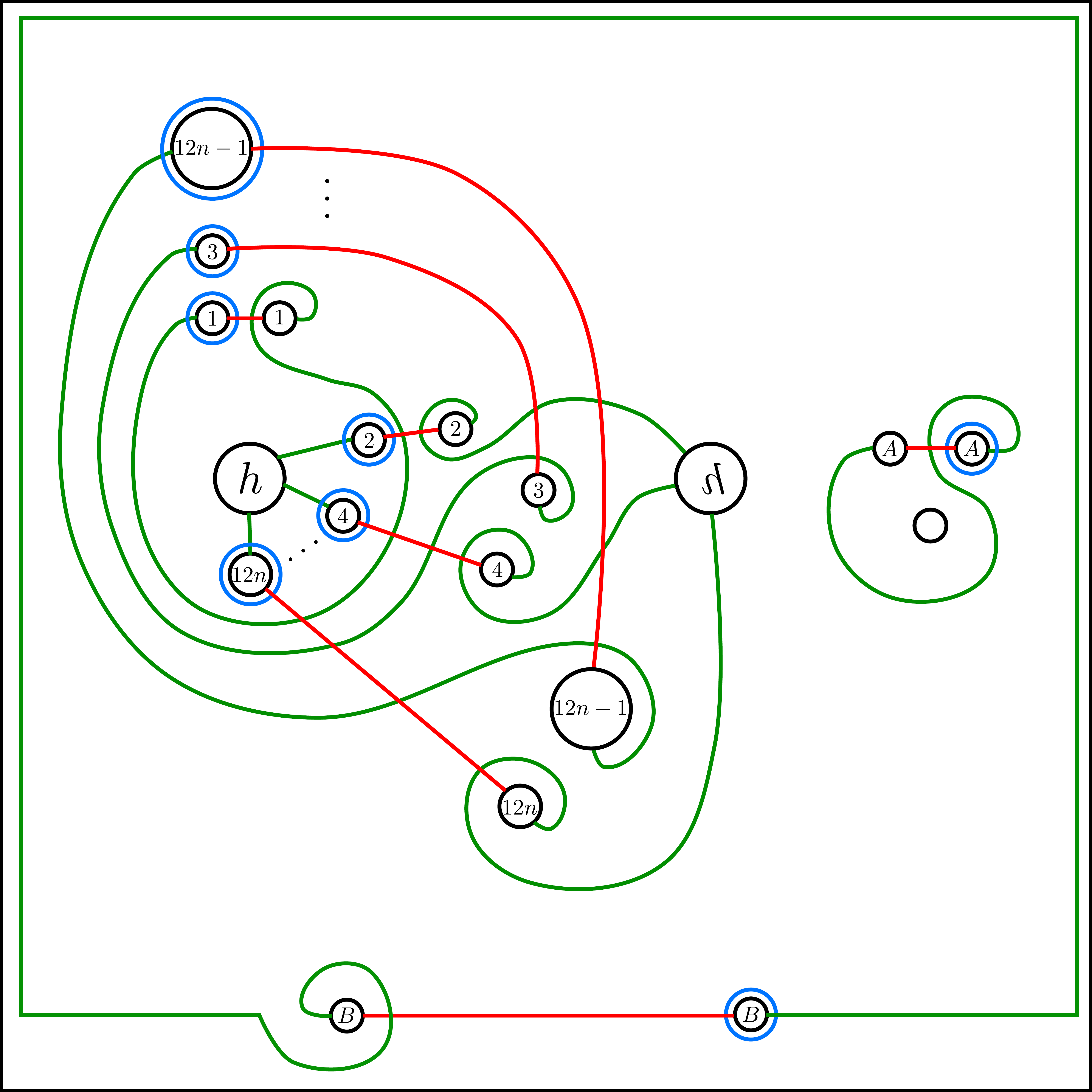}
    \setlength{\captionmargin}{1pt}
    \caption{A $(12n+3,3;1,2)$-relative trisection diagram of $E(n)-T^2 \times D^2$ with monodromy $(t_b t_a)^{6n} \circ t_{\delta_1} \circ t_{\delta_2}^{-1} \in \operatorname{Mod}(\Sigma_{p,2})$.}
    \label{fig:E(n)-T2D2_rtd}
  \end{minipage}
  \hfill
  \begin{minipage}{0.45\textwidth}
    \centering
    \includegraphics[width=\textwidth]{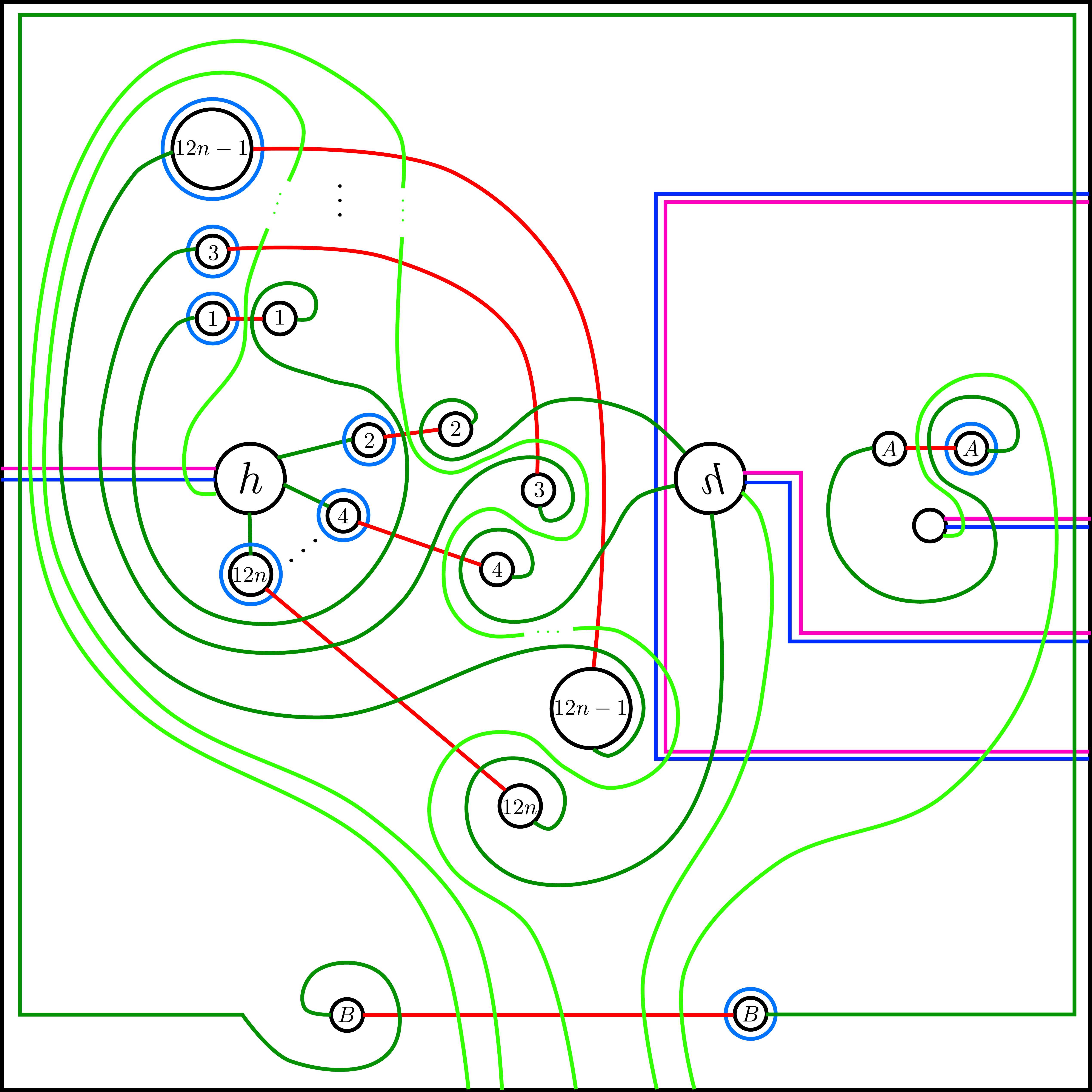}
    \setlength{\captionmargin}{5pt}
    \caption{An arced relative trisection diagram of $E(n)-T^2 \times D^2$ obtained from Figure \ref{fig:E(n)-T2D2_rtd}.}
    \label{fig:E(n)-T2D2_artd}
  \end{minipage}
\end{figure}

\begin{figure}[htbp]
  \centering
  \begin{minipage}{0.45\textwidth}
    \centering
    \includegraphics[width=\textwidth]{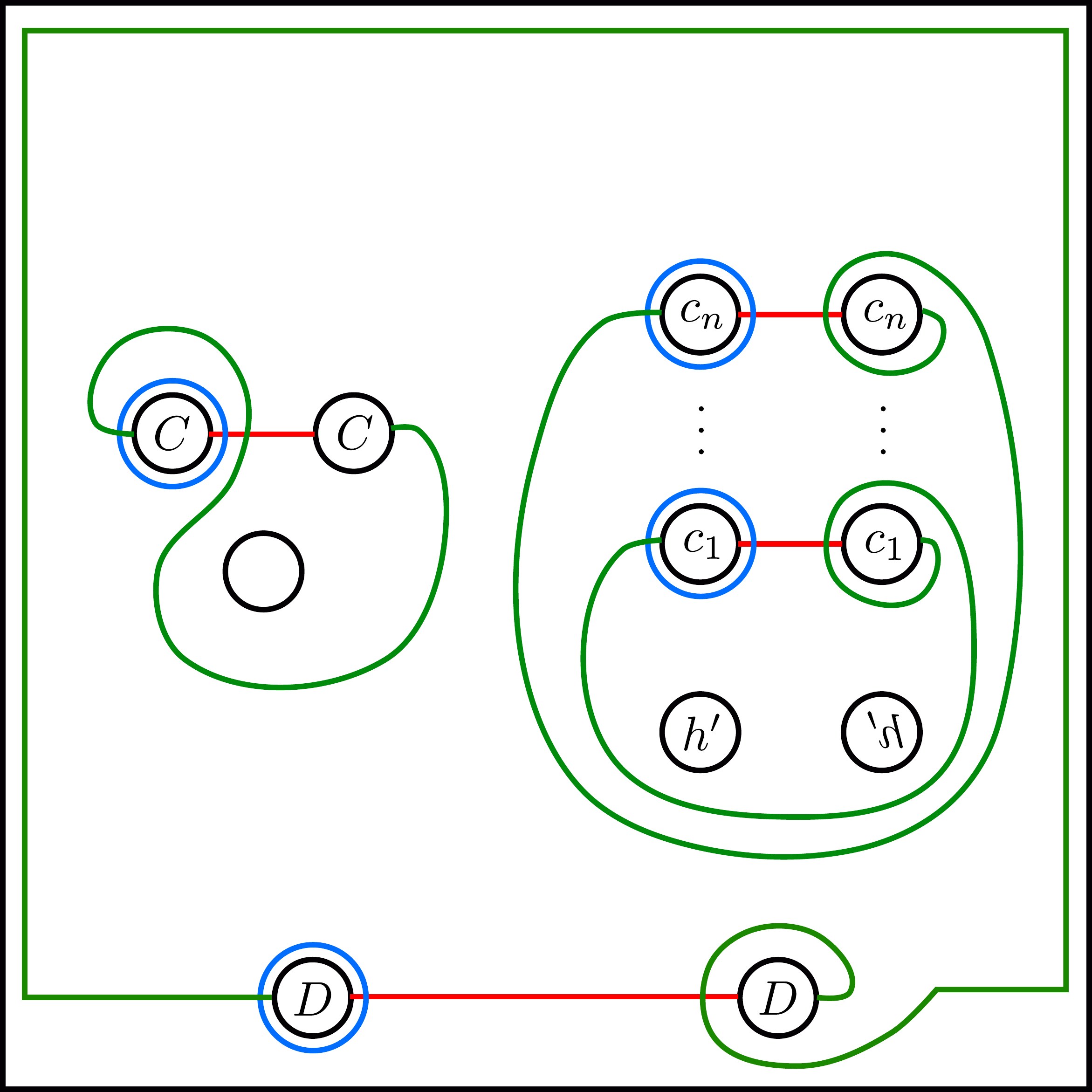}
    \setlength{\captionmargin}{2pt}
    \caption{A $(n+3,3;1,2)$-relative trisection diagram of $T^2 \times D^2 \# n \mathbb{C}P^2$ with monodromy $t_{\delta_3}^{-n} \circ t_{\delta_1}^{-1} \circ t_{\delta_2} \in \operatorname{Mod}(\Sigma_{p,2})$.}
    \label{fig:T2D2nCP2_rtd}
  \end{minipage}
  \hfill
  \begin{minipage}{0.45\textwidth}
    \centering
    \includegraphics[width=\textwidth]{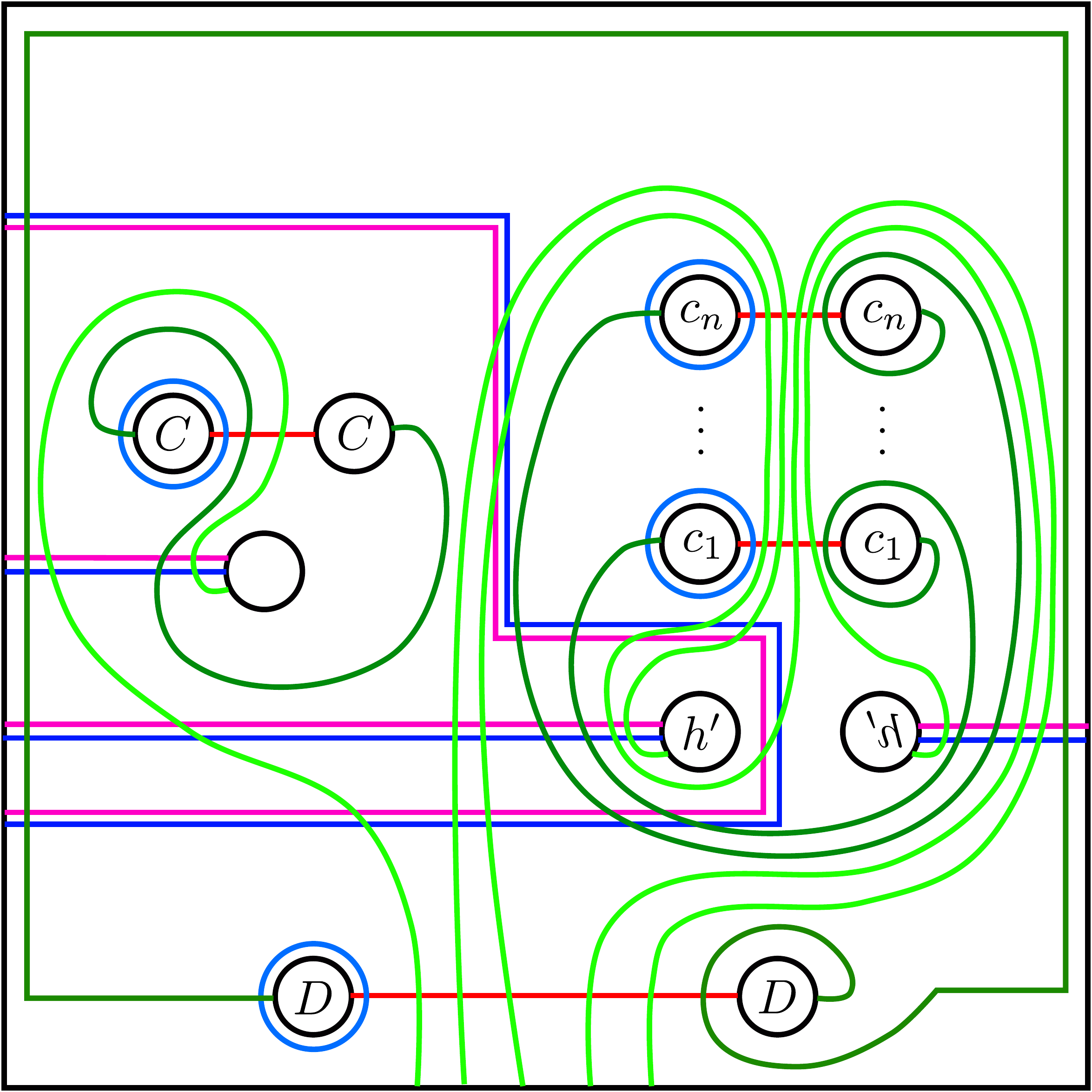}
    \setlength{\captionmargin}{5pt}
    \caption{An arced relative trisection diagram of $T^2 \times D^2 \# n \mathbb{C}P^2$ obtained from Figure \ref{fig:T2D2nCP2_rtd}.}
    \label{fig:T2D2nCP2_artd}
  \end{minipage}
\end{figure}

\begin{figure}[h]
\begin{center}
\includegraphics[width=12.5cm, height=12.5cm, keepaspectratio, scale=1]{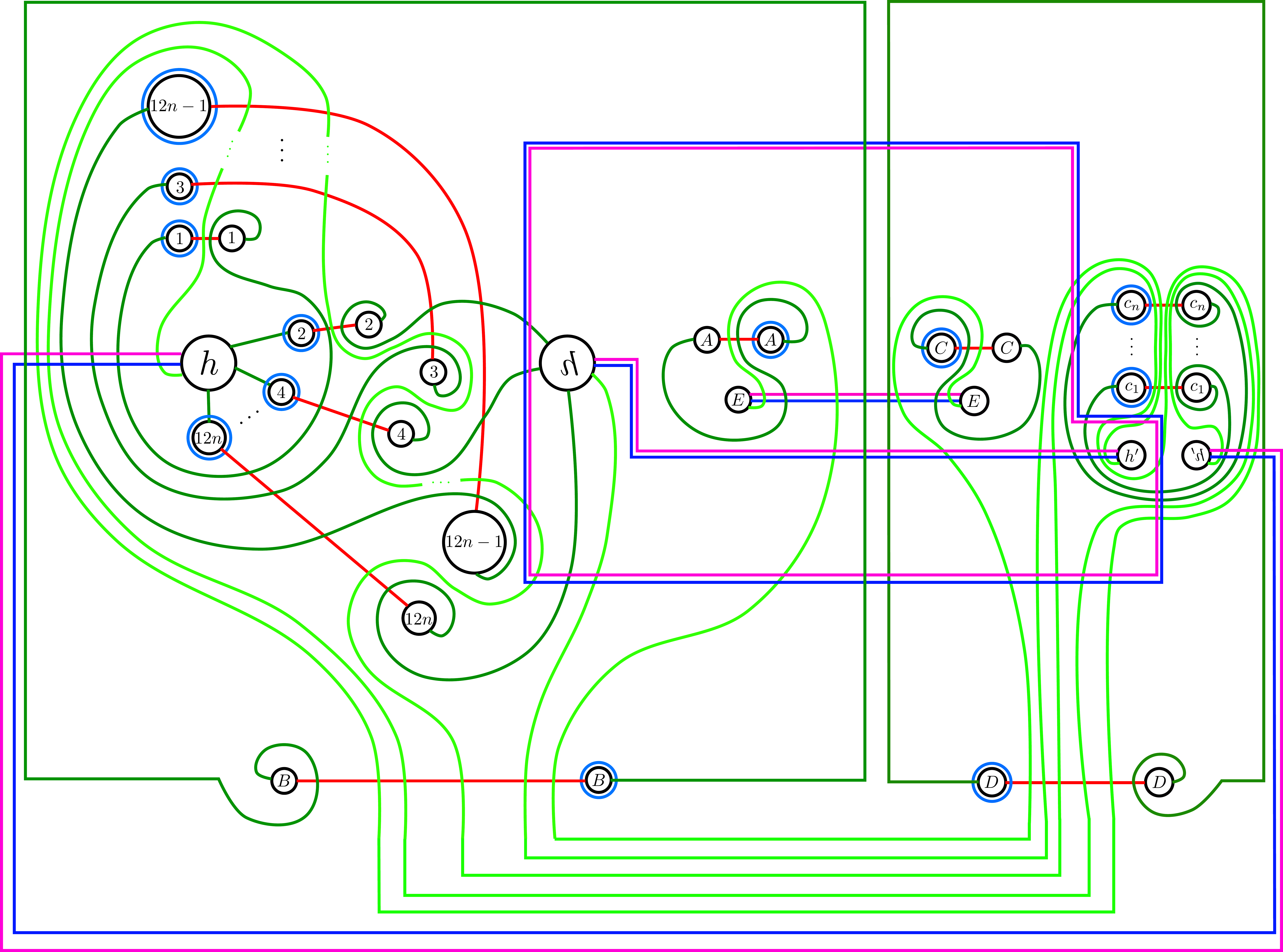}
\end{center}
\setlength{\captionmargin}{50pt}
\caption{A $(13n+7,3)$-trisection diagram of $E(n) \# n \mathbb{C}P^2 \cong (3n-1)\mathbb{C}P^2 \# (10n-1)\overline{\mathbb{C}P^2}$.}
\label{fig:E(n)nCP2}
\end{figure}

\begin{exm}\label{exm:E(1)_2,3}
Matsumoto~\cite{Matsumoto} constructed a genus-$2$ Lefschetz fibration over $S^2$. 
This was later generalized by Korkmaz~\cite{Korkmaz} and Gurtas~\cite{Gurtas} to higher genera. 
In particular, Gurtas provided new positive factorizations in mapping class groups, and computed the topology and the signature of the resulting Lefschetz fibrations. 

Using this framework, Fintushel and Stern~\cite{FintushelStern} described the knot-surgered $4$-manifold $E(n)_K$ as a twisted fiber sum of two copies of a genus-$(2h+n-1)$ Lefschetz fibration $f_{n,h} \colon M(n,h) \;\to\; S^2$ given by Gurtas, where $M(n,h)$ is diffeomorphic to $(S^2\times \Sigma_h)\#4n\overline{\mathbb{C}P^2}$.

Yun~\cite{Yun2006} obtained an explicit monodromy factorization of $f_{n,h}\colon M(n,h) \rightarrow S^2$ as follows;
\[
\eta_{n-1,h}^2=t_{\delta}\in \operatorname{Mod}(\Sigma_{2h+n-1,1}).
\]
Here, $\eta_{n-1,h}$ is defined by 
\[
\eta_{n-1,h}=t_{c_{2n-2}}t_{c_{2n-3}}\dots t_{c_{2}}t_{c_{1}}t_{c_{1}}t_{c_{2}}\dots t_{c_{2n-3}}t_{c_{2n-2}}t_{D_{0}}t_{D_{1}}\dots t_{D_{2h}}t_{c_{2n-1}},
\]
where the simple closed curves $\delta, c_1, \dots, c_{2n-1}, D_0, \dots, D_{2h}$ are the ones shown in Figure \ref{scc0_2}. 

\begin{figure}[h]
\begin{center}
\includegraphics[width=12cm,page=1]{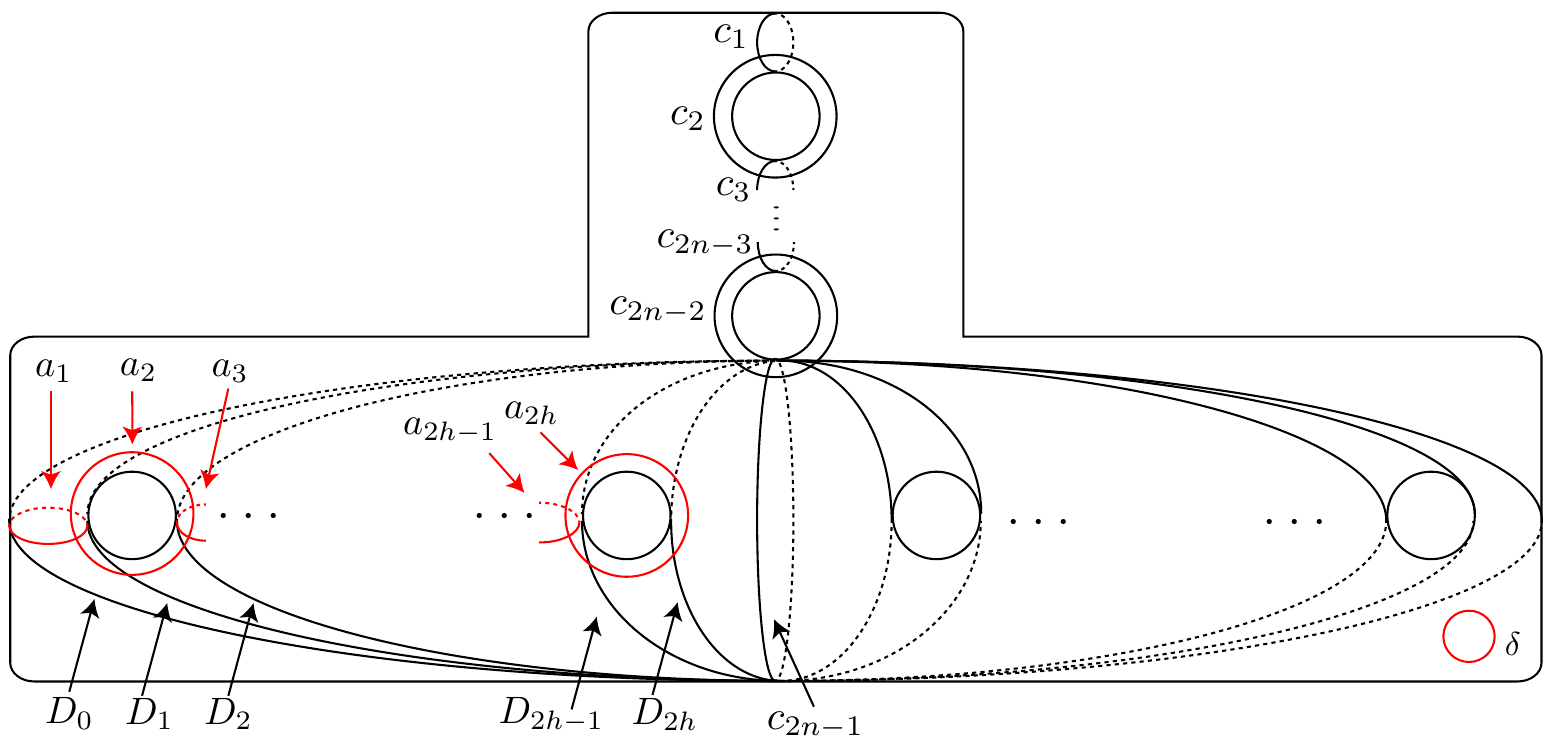}
\end{center}
\setlength{\captionmargin}{49pt}
\caption{The boundary $\delta$ and the simple closed curves $c_1, \dots, c_{2n-1}, D_0, \dots, D_{2h}, a_1, \dots, a_{2h}$ on $\Sigma_{2h+n-1,1}$.}
\label{scc0_2}
\end{figure}

Note that the relation above lives in 
$\operatorname{Mod}(\Sigma_{2h+n-1,1})$ and the Dehn twist $t_\delta$ along the boundary component is non–trivial.
After capping off the boundary component corresponding to $\delta$,
the twist $t_\delta$ becomes trivial in 
$\operatorname{Mod}(\Sigma_{2h+n-1})$.
Therefore the global monodromy of the Lefschetz fibration
over \(S^2\) yields a factorization of the identity:
\[
\eta_{n-1,h}^2
= \operatorname{id} \in \operatorname{Mod}(\Sigma_{2h+n-1}).
\]
Moreover, this Lefschetz fibration admits a section of self-intersection number $-2$.
Indeed, the relation $\eta_{n-1,h}^2 = t_{\delta} \in \operatorname{Mod}(\Sigma_{2h+n-1,1})$
shows that the boundary twist $t_{\delta}$ appears with exponent $+1$,
which corresponds to a section of self-intersection number $-1$.
After taking the twisted fiber sum of two copies, the resulting Lefschetz fibration
admits a section of self-intersection number $-2$.

Yun~\cite{Yun} also gave an explicit monodromy factorization for the 
Fintushel--Stern Lefschetz fibration
\[
f_{n,h}\#_{\Phi_K}f_{n,h} \colon M(n,h)\#_{\Phi_K} M(n,h) \to S^2,
\]
where $M(n,h)\#_{\Phi_K} M(n,h)$ is diffeomorphic to $E(n)_K$.
In particular, when $K$ is a fibered knot of genus $h$, 
$E(n)_K$ admits a genus-$(2h+n-1)$ Lefschetz fibration over $S^2$.
We define
\[
\Phi_K=(\varphi_K\# \operatorname{id}\#\operatorname{id})\times \operatorname{id}_{S^1}
:(\Sigma_h\#\Sigma_{n-1}\#\Sigma_h)\times S^1\to
(\Sigma_h\#\Sigma_{n-1}\#\Sigma_h)\times S^1,
\]
where $\varphi_K\in\operatorname{Mod}(\Sigma_{h,1})$ is a monodromy of the genus-$h$
fibered knot $K$ such that
\[
S^3\setminus \nu(K)\cong (\Sigma_{h,1}\times I)/((x,1)\sim(\varphi_K(x),0)).
\]
Here we define $\varphi_K \# \operatorname{id} \# \operatorname{id} \in
\operatorname{Mod}(\Sigma_h \# \Sigma_{n-1} \# \Sigma_h)$
to be the mapping class represented by a diffeomorphism obtained by extending
$\varphi_K$ by the identity outside a subsurface identified with $\Sigma_h$.
Although $\Phi_K$ is defined as a diffeomorphism of
$(\Sigma_h\#\Sigma_{n-1}\#\Sigma_h)\times S^1$,
its action on the monodromy factorization is given by the restriction
to the page, namely by the mapping class
$\varphi_K \# \operatorname{id} \# \operatorname{id} \in \operatorname{Mod}(\Sigma_{2h+n-1,1})$.
The diffeomorphism $\Phi_K$ determines an action on
$\operatorname{Mod}(\Sigma_{2h+n-1,1})$,
which on Dehn twists is given by sending
$t_\gamma$ to $t_{(\varphi_K\#\operatorname{id}\#\operatorname{id})(\gamma)}$.
This action is supported on the first $\Sigma_h$-summand.

With this notation, the global monodromy of the Lefschetz fibration on $E(n)_K$
can be written as
\[
\Phi_K(\eta_{n-1,h})^2\cdot\eta_{n-1,h}^2
= t_{\delta}^2
\in \operatorname{Mod}(\Sigma_{2h+n-1,1}).
\]

If $n=1$ and $K$ is the trefoil knot $T(2,3)$, the monodromy is 
\[
\Phi_{T(2,3)} := t_{a_2}^{-1}t_{a_1}^{-1}\in \operatorname{Mod}(\Sigma_2),
\]
where simple closed curves $a_1$ and $a_2$ are shown in Figure~\ref{scc0_2}. 
The Lefschetz fibration $f_{2,3}\#_{\Phi_{T(2,3)}}f_{2,3}\colon E(1)_{T(2,3)} \to S^2$
whose global monodromy is 
\[
(t_{D_0}t_{D_1}t_{D_2}t_{c_1})^2 
(t_{\Phi_{T(2,3)}(D_0)}t_{\Phi_{T(2,3)}(D_1)}t_{\Phi_{T(2,3)}(D_2)}t_{\Phi_{T(2,3)}(c_1)})^2=t_{\delta}^{2} 
\in \operatorname{Mod}(\Sigma_{2,1})
\]
is built from the twisted fiber sum of two copies of the Matsumoto fibration. 
The corresponding simple closed curves $D_0,$ $D_1,$ $D_2,$ $c_1,$ $\Phi_{T(2,3)}(D_0),$ $\Phi_{T(2,3)}(D_1),$ $\Phi_{T(2,3)}(D_2),$ $\Phi_{T(2,3)}(c_1)$ are shown in Figure~\ref{scc5}.
The total space $E(1)_{T(2,3)}$ is diffeomorphic to $E(1)_{2,3}$, which is an exotic copy of $E(1)$ known as the Dolgachev surface. 

We can draw the relative trisection diagram of $E(1)_{2,3}-\Sigma_2 \times D^2$ constructed from this Lefschetz fibration as in Figure \ref{fig:E(1)_2,3_rtd}. Thus, from Theorem \ref{thm:-n-section}, we can construct a $(27,5)$-trisection diagram of $E(1)_{2,3} \# 2\mathbb{C}P^2 \cong 3\mathbb{C}P^2 \# 9\overline{\mathbb{C}P^2}$ as in Figure \ref{fig:E(1)_2,3_ 2CP^2}. See Figure \ref{fig:shoryaku} for the notation in Figures \ref{fig:E(1)_2,3_rtd}, \ref{fig:E(1)_2,3_artd} and \ref{fig:E(1)_2,3_ 2CP^2}.
Similarly, we can construct the trisection diagram of $E(n)_K\# 2\mathbb{C}P^2 \cong (2n+1)\mathbb{C}P^2 \# (10n-1)\overline{\mathbb{C}P^2}$ for any 2-bridge fibered knot $K$.
\end{exm}

\begin{figure}[h]
\begin{center}
\includegraphics[width=13cm,page=1]{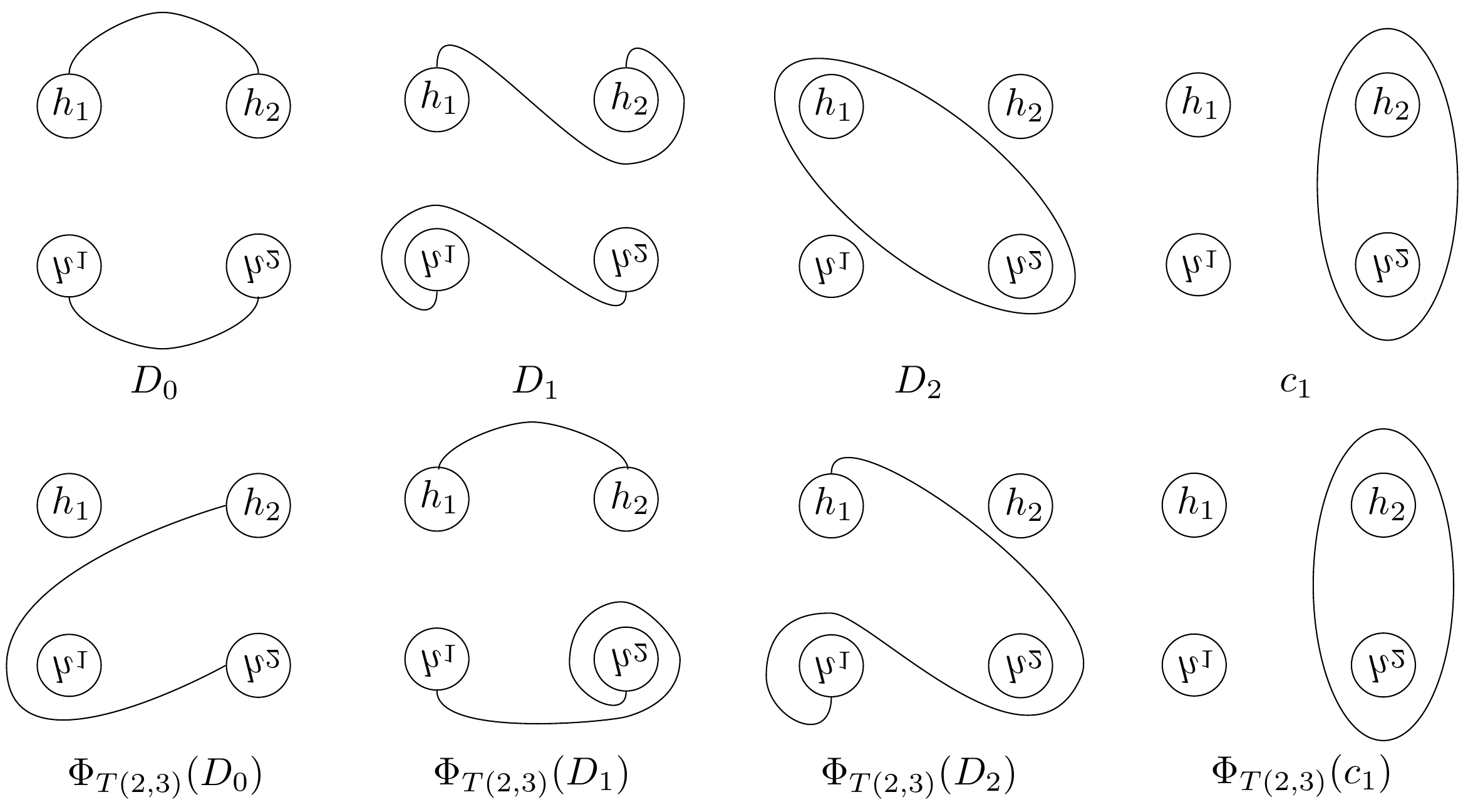}
\end{center}
\setlength{\captionmargin}{48pt}
\caption{The simple closed curves $D_0,$ $D_1,$ $D_2,$ $c_1,$ $\Phi_{T(2,3)}(D_0),$ $\Phi_{T(2,3)}(D_1),$ $\Phi_{T(2,3)}(D_2),$ $\Phi_{T(2,3)}(c_1)$ on $\Sigma_{2}$.}
\label{scc5}
\end{figure}

\begin{figure}[h]
\begin{center}
\includegraphics[width=12cm, height=12cm, keepaspectratio, scale=1]{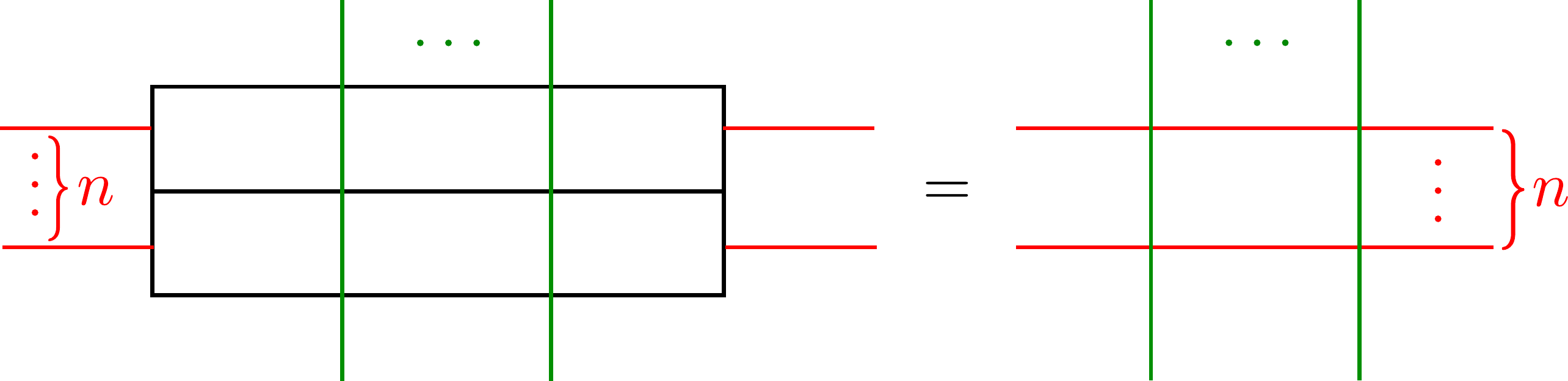}
\end{center}
\setlength{\captionmargin}{50pt}
\caption{A notation for $\alpha$ curves of trisection diagrams in Figures \ref{fig:E(1)_2,3_rtd}, \ref{fig:E(1)_2,3_artd} and \ref{fig:E(1)_2,3_ 2CP^2}.}
\label{fig:shoryaku}
\end{figure}

\begin{figure}[h]
\begin{center}
\includegraphics[width=18cm, height=20cm, keepaspectratio, scale=1]{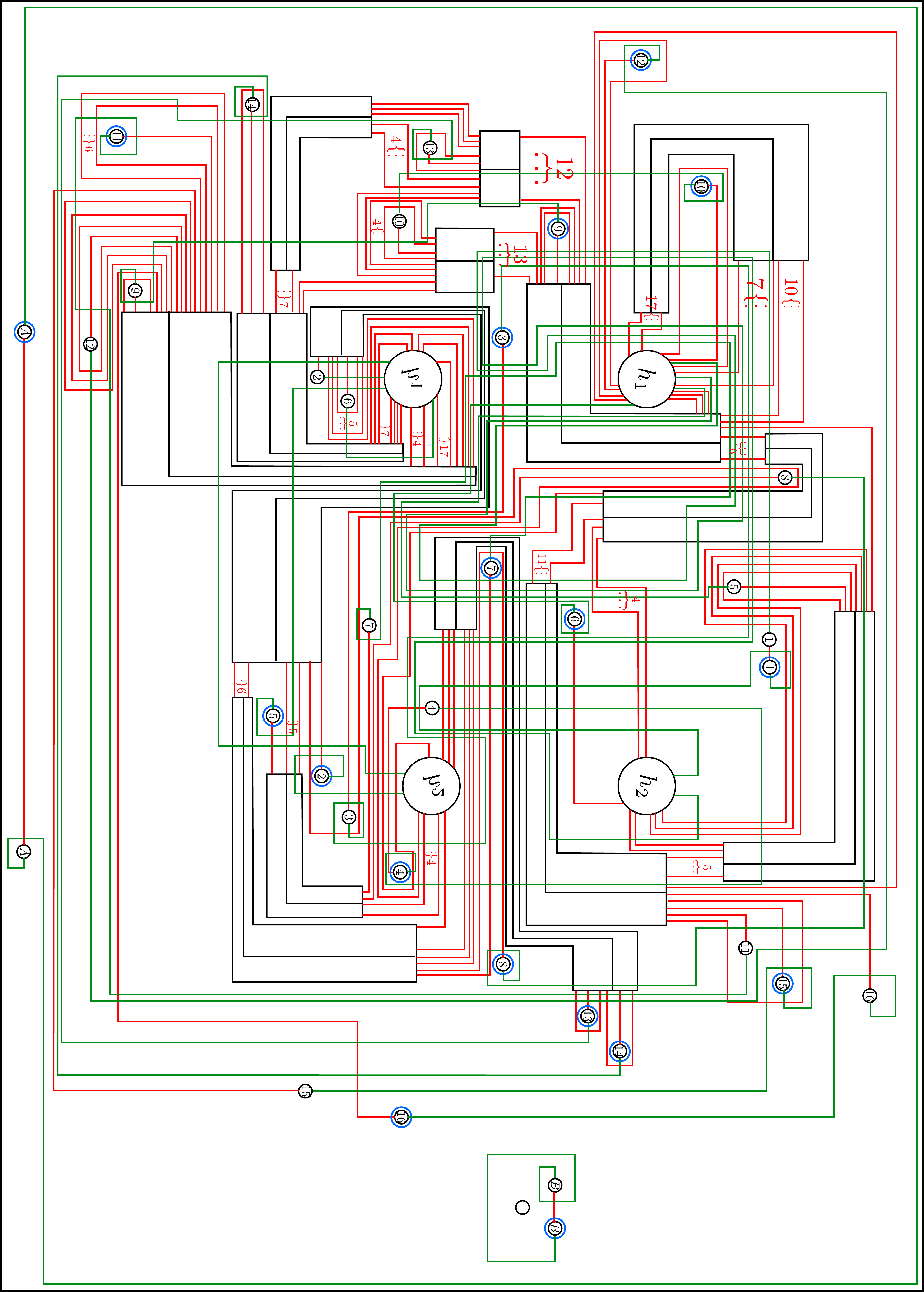}
\end{center}
\setlength{\captionmargin}{60pt}
\caption{A $(20,5;2,2)$-relative trisection diagram of $E(1)_{2,3}-\Sigma_2 \times D^2$ with monodromy $(t_{D_0}t_{D_1}t_{D_2}t_{c_1})^2 
(t_{\Phi_{T(2,3)}(D_0)}t_{\Phi_{T(2,3)}(D_1)}t_{\Phi_{T(2,3)}(D_2)}t_{\Phi_{T(2,3)}(c_1)})^2 \circ t_{\delta_1} \circ t_{\delta_2}^{-1} \in \operatorname{Mod}(\Sigma_{p,2})$, where $\Phi_{T(2,3)} = t_{a_2}^{-1}t_{a_1}^{-1}$.}
\label{fig:E(1)_2,3_rtd}
\end{figure}

\begin{figure}[h]
\begin{center}
\includegraphics[width=18cm, height=20cm, keepaspectratio, scale=1]{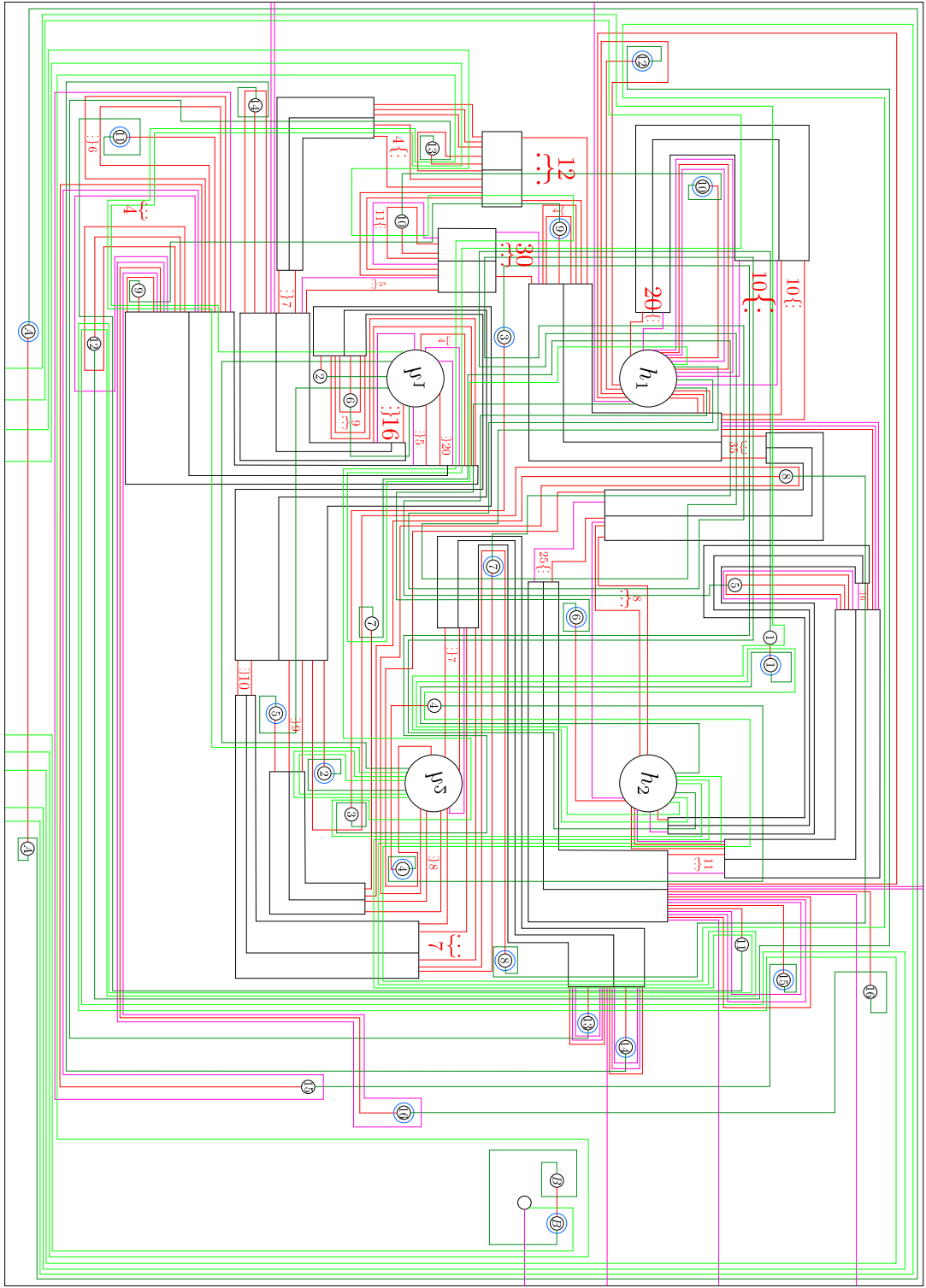}
\end{center}
\setlength{\captionmargin}{50pt}
\caption{An arced relative trisection diagram of $E(1)_{2,3}-\Sigma_2 \times D^2$ obtained from Figure \ref{fig:E(1)_2,3_rtd}. The blue arcs parallel to the pink arcs are omitted.}
\label{fig:E(1)_2,3_artd}
\end{figure}

\begin{figure}[h]
\begin{center}
\includegraphics[width=20cm, height=20cm, keepaspectratio, scale=1]{E_1__2,3__2CP_2.pdf}
\end{center}
\setlength{\captionmargin}{50pt}
\caption{A $(27,5)$-trisection diagram of $E(1)_{2,3}\#2\mathbb{C}P^2 \cong 3\mathbb{C}P^2 \# 9 \overline{\mathbb{C}P^2}$. The blue curves parallel to the pink curves are omitted.}
\label{fig:E(1)_2,3_ 2CP^2}
\end{figure}

\begin{que}\label{que:CP^2}
Can we explicitly construct trisection diagrams of 4-manifolds admitting achiral Lefschetz fibrations over $S^2$ with a $(-n)$-section from Theorem \ref{thm:-n-section} by removing $\abs{n}$ connected summands of $\mathbb{C}P^2$ or $\overline{\mathbb{C}P^2}$? In particular, what about exotic 4-manifolds such as $E(1)_{2,3}$ in Example \ref{exm:E(1)_2,3}?
\end{que}

\begin{que}\label{que:stabili}
Is a trisection diagram constructed in Corollary \ref{cor:ACD} related to the standard trisection diagram of $k \mathbb{C}P^2 \# \ell \overline{\mathbb{C}P^2}$ shown in Figure \ref{fig:CP^2s} by surface diffeomorphisms, handle slides among the same family curves and destabilizations? In particular, what about trisection diagrams in Examples \ref{exm:E(n)} and \ref{exm:E(1)_2,3}?
\end{que}

Question \ref{que:stabili} is an analogue of Conjecture \ref{conj:4DWC} for $k \mathbb{C}P^2 \# \ell \overline{\mathbb{C}P^2}$. It is not known whether there exists a simply-conected closed 4-manifold admitting a non-isotopic trisection. On the other hand, Islambouli \cite{MR4308281} showed that there exist non-simply conected closed 4-manifolds admitting non-isotopic trisections.

\begin{rem}
In Theorem \ref{thm:-n-section}, if $n=1$, we can construct a trisection diagram of $X \# \mathbb{C}P^2$ for a closed 4-manifold $X$ admitting an achiral Lefschetz fibration with a $(-1)$-section. On the other hand, Castro and Ozbagci \cite{MR3999550} constructed a trisection diagram of the 4-manifold $X$ itself by using the Lefschetz fibration. Thus, we can obtain another triseciton diagram of $X \# \mathbb{C}P^2$ by taking the connected sum of the trisection diagram of $X$ with a genus-1 trisection diagram of $\mathbb{C}P^2$ such as Figure \ref{fig:CP^2}. It is not known whether these two trisection diagrams of $X \# \mathbb{C}P^2$ are related without stabilizations, namely, related by surface diffeomorphisms, handle slides among the same family curves and destabilizations.
\end{rem}

\bibliographystyle{amsalpha}
\bibliography{trisection, ACD, LF}

@article{zbMATH07206908,
 author = {Choi, Hakho and Park, Jongil and Yun, Ki-Heon},
 title = {On dissolving knot surgery 4-manifolds under a {{\(\mathbb{CP}^2\)}}-connected sum},
 fjournal = {The Asian Journal of Mathematics},
 journal = {Asian J. Math.},
 issn = {1093-6106},
 volume = {23},
 number = {5},
 pages = {735--748},
 year = {2019},
 language = {English},
 doi = {10.4310/AJM.2019.v23.n5.a2},
 zbMATH = {7206908},
 Zbl = {1457.57026}
}

@article{zbMATH03505878,
 author = {Mandelbaum, Richard and Moishezon, Boris},
 title = {On the topological structure of non-singular algebraic surfaces in {{\(CP^3\)}}},
 fjournal = {Topology},
 journal = {Topology},
 issn = {0040-9383},
 volume = {15},
 pages = {23--40},
 year = {1976},
 language = {English},
 doi = {10.1016/0040-9383(76)90047-1},
 zbMATH = {3505878},
 Zbl = {0323.57005}
}

@book{zbMATH03608421,
 author = {Moishezon, Boris},
 title = {Complex surfaces and connected sums of complex projective planes},
 fseries = {Lecture Notes in Mathematics},
 series = {Lect. Notes Math.},
 issn = {0075-8434},
 volume = {603},
 year = {1977},
 publisher = {Springer, Cham},
 language = {English},
 zbMATH = {3608421},
 Zbl = {0392.32015}
}

@article{zbMATH03729115,
 author = {Mandelbaum, Richard and Moishezon, Boris},
 title = {On the topology of simply-connected algebraic surfaces},
 fjournal = {Transactions of the American Mathematical Society},
 journal = {Trans. Am. Math. Soc.},
 issn = {0002-9947},
 volume = {260},
 pages = {195--222},
 year = {1980},
 language = {English},
 doi = {10.2307/1999883},
 zbMATH = {3729115},
 Zbl = {0465.57014}
}

@misc{zbMATH00016235,
 author = {Gompf, Robert E.},
 title = {On the topology of algebraic surfaces},
 year = {1990},
 language = {English},
 howpublished = {Geometry of low-dimensional manifolds. 1: {Gauge} theory and algebraic surfaces, {Proc}. {Symp}., {Durham}/{UK} 1989, {Lond}. {Math}. {Soc}. {Lect}. {Note} {Ser}. 150, 41-54 (1990).},
 zbMATH = {16235},
 Zbl = {0837.57026}
}

@article{Matsumoto,
  author    = {Yukio Matsumoto},
  title     = {Lefschetz fibrations of genus two --- a topological approach},
  journal   = {Topology Appl.},
  volume    = {69},
  number    = {1},
  pages     = {1--19},
  year      = {1996},
  doi       = {10.1016/0166-8641(95)00048-8}
}

@article{FintushelStern,
  author    = {Ronald Fintushel and Ronald J. Stern},
  title     = {Knots, links, and 4-manifolds},
  journal   = {Invent. Math.},
  volume    = {134},
  number    = {2},
  pages     = {363--400},
  year      = {1998},
  doi       = {10.1007/s002220050268}
}

@article{Gompf98,
  author    = {Robert E. Gompf},
  title     = {Handlebody construction of {Stein} surfaces},
  journal   = {Ann. of Math. (2)},
  volume    = {148},
  number    = {2},
  pages     = {619--693},
  year      = {1998},
  doi       = {10.2307/121005}
}

@book{GS99,
  author    = {Robert E. Gompf and Andr{\'a}s I. Stipsicz},
  title     = {4-Manifolds and {Kirby} Calculus},
  series    = {Graduate Studies in Mathematics},
  volume    = {20},
  publisher = {Amer. Math. Soc.},
  address   = {Providence, RI},
  year      = {1999}
}

@article{AO01,
  author  = {Selman Akbulut and Burak Ozbagci},
  title   = {Lefschetz fibrations on compact {Stein} surfaces},
  journal = {Geom. Topol.},
  volume  = {5},
  pages   = {319--334},
  year    = {2001},
  doi     = {10.2140/gt.2001.5.319}
}

@article{Korkmaz,
  author    = {Mustafa Korkmaz},
  title     = {Noncomplex smooth 4-manifolds with {L}efschetz fibrations},
  journal   = {Int. Math. Res. Not.},
  volume    = {2001},
  number    = {3},
  pages     = {115--128},
  year      = {2001},
  doi       = {10.1155/S107379280100006X}
}

@inproceedings{Giroux02,
  author    = {Emmanuel Giroux},
  title     = {G{\'e}om{\'e}trie de contact: de la dimension trois vers les dimensions sup{\'e}rieures},
  booktitle = {Proceedings of the International Congress of Mathematicians (Beijing, 2002), Vol. II},
  publisher = {Higher Education Press},
  pages     = {405--414},
  year      = {2002}
}

@article{Gurtas,
  author    = {Yusuf Z. Gurtas},
  title     = {Positive {Dehn} twist expressions for some elements of the mapping class group of a surface},
  journal   = {Rocky Mountain J. Math.},
  volume    = {39},
  number    = {3},
  pages     = {989--1004},
  year      = {2009},
  doi       = {10.1216/RMJ-2009-39-3-989}
}

@article{Yun2006,
  author    = {Ki-Heon Yun},
  title     = {On the signature of a {L}efschetz fibration coming from an involution},
  journal   = {Math. Ann.},
  volume    = {336},
  number    = {2},
  pages     = {385--398},
  year      = {2006},
  doi       = {10.1007/s00208-006-0008-9}
}

@article{Yun,
  author    = {Ki-Heon Yun},
  title     = {Monodromies of {F}intushel--{S}tern's construction on knot surgered elliptic surfaces},
  journal   = {J. Knot Theory Ramifications},
  volume    = {21},
  number    = {13},
  pages     = {1250120},
  year      = {2012},
  doi       = {10.1142/S0218216512501207}
}

@article{Stipsicz1999,
  author  = {Stipsicz, Andr{\'a}s I.},
  title   = {Sections of {L}efschetz fibrations and {S}tein fillings},
  journal = {Turkish Journal of Mathematics},
  volume  = {23},
  number  = {1},
  pages   = {145--150},
  year    = {1999},
  issn    = {1300-0098},
  url     = {http://journals.tubitak.gov.tr/math/issues/mat-99-23-1/mat-23-1-16-9809-11.pdf}
}

@article{McDuff1990,
  author  = {McDuff, Dusa},
  title   = {The structure of rational and ruled symplectic 4-manifolds},
  journal = {Journal of the American Mathematical Society},
  volume  = {3},
  number  = {3},
  pages   = {679--712},
  year    = {1990},
  doi     = {10.1090/S0894-0347-1990-1057041-7}
}

@inproceedings{Ehr51,
  author    = {Ehresmann, Charles},
  title     = {Les connexions infinit{\'e}simales dans un espace fibr{\'e} diff{\'e}rentiable},
  booktitle = {Colloque de Topologie},
  address   = {Bruxelles},
  year      = {1951}
}

@article {MR1555438,
    AUTHOR = {Dehn, M.},
     TITLE = {Die {G}ruppe der {A}bbildungsklassen},
      NOTE = {Das arithmetische Feld auf Fl\"achen},
   JOURNAL = {Acta Math.},
  FJOURNAL = {Acta Mathematica},
    VOLUME = {69},
      YEAR = {1938},
    NUMBER = {1},
     PAGES = {135--206},
      ISSN = {0001-5962,1871-2509},
   MRCLASS = {99-04},
  MRNUMBER = {1555438},
       DOI = {10.1007/BF02547712},
       URL = {https://doi-org.kras.lib.keio.ac.jp/10.1007/BF02547712},
}

@book {MR881797,
    AUTHOR = {Dehn, Max},
     TITLE = {Papers on group theory and topology},
      NOTE = {Translated from the German and with introductions and an
              appendix by John Stillwell,
              With an appendix by Otto Schreier},
 PUBLISHER = {Springer-Verlag, New York},
      YEAR = {1987},
     PAGES = {viii+396},
      ISBN = {0-387-96416-9},
   MRCLASS = {01A75 (20-03 55-03 57-03)},
  MRNUMBER = {881797},
MRREVIEWER = {J.\ S.\ Birman},
       DOI = {10.1007/978-1-4612-4668-8},
       URL = {https://doi-org.kras.lib.keio.ac.jp/10.1007/978-1-4612-4668-8},
}

@phdthesis{castro2016relative,
  title={Relative trisections of smooth 4-manifolds with boundary},
  author={Castro, Nickolas Andres},
  year={2016},
  school={University of Georgia},
}

@article {MR3770114,
    AUTHOR = {Castro, Nickolas A. and Gay, David T. and Pinz\'{o}n-Caicedo,
              Juanita},
     TITLE = {Diagrams for relative trisections},
   JOURNAL = {Pacific J. Math.},
  FJOURNAL = {Pacific Journal of Mathematics},
    VOLUME = {294},
      YEAR = {2018},
    NUMBER = {2},
     PAGES = {275--305},
     }

@article {MR3999550,
    AUTHOR = {Castro, Nickolas A. and Ozbagci, Burak},
     TITLE = {Trisections of 4-manifolds via {L}efschetz fibrations},
   JOURNAL = {Math. Res. Lett.},
  FJOURNAL = {Mathematical Research Letters},
    VOLUME = {26},
      YEAR = {2019},
    NUMBER = {2},
     PAGES = {383--420},
      ISSN = {1073-2780},
   MRCLASS = {57K40},
  MRNUMBER = {3999550},
MRREVIEWER = {Alexander Zupan},
       DOI = {10.4310/MRL.2019.v26.n2.a3},
       URL = {https://doi.org/10.4310/MRL.2019.v26.n2.a3},
}

@article {MR3590351,
    AUTHOR = {Gay, David and Kirby, Robion},
     TITLE = {Trisecting 4-manifolds},
   JOURNAL = {Geom. Topol.},
  FJOURNAL = {Geometry \& Topology},
    VOLUME = {20},
      YEAR = {2016},
    NUMBER = {6},
     PAGES = {3097--3132},
      memo = {bibtexの情報を全て引用すると，表示されたのはauthor, title, journal, volume, year, number, pages, mrnumberだけだった．
      特にURLは表示されなかった}, 
}

@article {MR4308281,
    AUTHOR = {Islambouli, Gabriel},
     TITLE = {Nielsen equivalence and trisections},
   JOURNAL = {Geom. Dedicata},
  FJOURNAL = {Geometriae Dedicata},
    VOLUME = {214},
      YEAR = {2021},
     PAGES = {303--317},
      ISSN = {0046-5755},
   MRCLASS = {57K40},
  MRNUMBER = {4308281},
MRREVIEWER = {Richard Stong},
       DOI = {10.1007/s10711-021-00617-y},
       URL = {https://doi.org/10.1007/s10711-021-00617-y},
}

@article {MR3544545,
    AUTHOR = {Meier, Jeffrey and Schirmer, Trent and Zupan, Alexander},
     TITLE = {Classification of trisections and the generalized property {R}
              conjecture},
   JOURNAL = {Proc. Amer. Math. Soc.},
  FJOURNAL = {Proceedings of the American Mathematical Society},
    VOLUME = {144},
      YEAR = {2016},
    NUMBER = {11},
     PAGES = {4983--4997},
     }

\end{document}